\documentclass[11pt]{amsart}

\usepackage{amsmath,amssymb,amsfonts,amscd}
\usepackage{url}
\usepackage{color}
\usepackage[usenames,dvipsnames]{xcolor}
\usepackage[all]{xy}
\usepackage{hyperref}
\usepackage{mathtools}
\usepackage{mathrsfs}
\usepackage{a4wide}
\usepackage{graphicx}

\renewcommand{\Im}{\operatorname{Im}}

\newcommand{\tree}{\mathcal{ST}}
\newcommand{\forest}{\mathcal{SF}_2}

\begin{document} 
\theoremstyle{plain}
\newtheorem{thm}{Theorem}[section]
\newtheorem{lem}[thm]{Lemma}
\newtheorem{cor}[thm]{Corollary}
\newtheorem{prop}[thm]{Proposition}
\newtheorem{prop-defn}[thm]{Proposition-Definition}
\newtheorem{question}[thm]{Question}
\newtheorem{claim}[thm]{Claim}

\theoremstyle{definition}
\newtheorem{remark}[thm]{Remark}
\newtheorem{defn}[thm]{Definition}
\newtheorem{ex}[thm]{Example}
\newtheorem{conj}[thm]{Conjecture}
\numberwithin{equation}{section}
\newcommand{\eq}[2]{\begin{equation}\label{#1}#2 \end{equation}}
\newcommand{\ml}[2]{\begin{multline}\label{#1}#2 \end{multline}}
\newcommand{\ga}[2]{\begin{gather}\label{#1}#2 \end{gather}}
\newcommand{\mc}{\mathcal}
\newcommand{\mb}{\mathbb}
\newcommand{\surj}{\twoheadrightarrow}
\newcommand{\inj}{\hookrightarrow}
\newcommand{\red}{{\rm red}}
\newcommand{\codim}{{\rm codim}}
\newcommand{\rank}{{\rm rank}}
\newcommand{\Pic}{{\rm Pic}}
\newcommand{\Div}{{\rm Div}}
\newcommand{\im}{{\rm im}}
\newcommand{\Spec}{{\rm Spec \,}}
\newcommand{\Sing}{{\rm Sing}}
\newcommand{\Char}{{\rm char}}
\newcommand{\Tr}{{\rm Tr}}
\newcommand{\Gal}{{\rm Gal}}
\newcommand{\Min}{{\rm Min \ }}
\newcommand{\Max}{{\rm Max \ }}
\newcommand{\ti}{\times }
\newcommand{\sA}{{\mathcal A}}
\newcommand{\sB}{{\mathcal B}}
\newcommand{\sC}{{\mathcal C}}
\newcommand{\sD}{{\mathcal D}}
\newcommand{\sE}{{\mathcal E}}
\newcommand{\sF}{{\mathcal F}}
\newcommand{\sG}{{\mathcal G}}
\newcommand{\sH}{{\mathcal H}}
\newcommand{\sI}{{\mathcal I}}
\newcommand{\sJ}{{\mathcal J}}
\newcommand{\sK}{{\mathcal K}}
\newcommand{\sL}{{\mathcal L}}
\newcommand{\sM}{{\mathcal M}}
\newcommand{\sN}{{\mathcal N}}
\newcommand{\sO}{{\mathcal O}}
\newcommand{\sP}{{\mathcal P}}
\newcommand{\sQ}{{\mathcal Q}}
\newcommand{\sR}{{\mathcal R}}
\newcommand{\sS}{{\mathcal S}}
\newcommand{\sT}{{\mathcal T}}
\newcommand{\sU}{{\mathcal U}}
\newcommand{\sV}{{\mathcal V}}
\newcommand{\sW}{{\mathcal W}}
\newcommand{\sX}{{\mathcal X}}
\newcommand{\sY}{{\mathcal Y}}
\newcommand{\sZ}{{\mathcal Z}}
\newcommand{\A}{{\mathbb A}}
\newcommand{\B}{{\mathbb B}}
\newcommand{\C}{{\mathbb C}}
\newcommand{\D}{{\mathbb D}}
\newcommand{\E}{{\mathbb E}}
\newcommand{\F}{{\mathbb F}}
\newcommand{\G}{{\mathbb G}}
\renewcommand{\H}{{\mathbb H}}
\newcommand{\I}{{\mathbb I}}
\newcommand{\J}{{\mathbb J}}
\newcommand{\M}{{\mathbb M}}
\newcommand{\N}{{\mathbb N}}
\renewcommand{\P}{{\mathbb P}}
\newcommand{\Q}{{\mathbb Q}}
\newcommand{\R}{{\mathbb R}}
\newcommand{\T}{{\mathbb T}}
\newcommand{\U}{{\mathbb U}}
\newcommand{\V}{{\mathbb V}}
\newcommand{\W}{{\mathbb W}}
\newcommand{\X}{{\mathbb X}}
\newcommand{\Y}{{\mathbb Y}}
\newcommand{\Z}{{\mathbb Z}}
\newcommand{\pic}{{\text{Pic}(C,\sD)[E,\nabla]}}
\newcommand{\ocd}{{\Omega^1_C\{\sD\}}}
\newcommand{\oc}{{\Omega^1_C}}
\newcommand{\al}{{\alpha}}
\newcommand{\be}{{\beta}}
\newcommand{\ta}{{\theta}}
\newcommand{\ve}{{\varepsilon}}
\newcommand{\lr}[2]{\langle #1,#2 \rangle}
\newcommand{\nnn}{\newline\newline\noindent}
\newcommand{\nn}{\newline\noindent}
\newcommand{\snote}[1]{{\color{blue} Spencer: #1}}

\newcommand{\onote}[1]{{\color{blue} Omid: #1}}
\newcommand{\jnote}[1]{{\color{orange} Javier: #1}}
\newcommand{\jinote}[1]{{\color{teal} Jose: #1}}
\newcommand{\vs}{\mathcal V} 
\newcommand{\ws}{\mathcal W} 
\newcommand{\spanningf}{\mathcal{SF}}
\newcommand{\gr}{\mathrm{gr}}

\newcommand{\ok}{R} 
\newcommand{\ovar}{Y} 
\newcommand{\Diag}{\mathrm{Diag}}
\newcommand{\green}{\mathfrak g}
\newcommand{\p}{\mathbf{p}}
\newcommand{\ps}{\mathbf{p}}
\newcommand{\oA}{\mathfrak A}
\newcommand{\oB}{\mathfrak B}
\newcommand{\oD}{\mathfrak D}

\newcommand{\ad}{\operatorname{ad}}
\newcommand{\K}{\mathbb K}
\newcommand{\an}{\operatorname{an}}
\newcommand{\val}{\mathrm{val}}
\newcommand{\Ar}{\mathrm{Ar}}
\newcommand{\onto}{\twoheadrightarrow}  
\newcommand{\oL}{\mathrm{L}}
\newcommand{\rl}{\mathit{l}}

\title[The exchange graph and variations of the ratio of Symanzik polynomials]{The exchange graph and variations of the ratio of the two Symanzik polynomials}
\author[Omid Amini]{Omid Amini}
\address{CNRS--DMA, \'Ecole Normale Sup\'erieure, 45 rue d'Ulm, Paris}
\email{oamini@math.ens.fr}
\date{July 2nd, 2016}
\begin{abstract}
Correlation functions in quantum field theory are calculated using Feynman amplitudes, which 
are finite dimensional integrals associated to graphs. The integrand is the exponential of 
the ratio of the first and second Symanzik polynomials associated to the Feynman graph, which are described in terms of the spanning trees and spanning 2-forests of the graph, respectively.  

In a previous paper with Bloch, Burgos and Fres\'an, we related this ratio to the asymptotic of the 
Archimedean height pairing between degree zero divisors on degenerating families of Riemann surfaces. 
Motivated by this, we consider in this paper the variation of the ratio of the two Symanzik polynomials 
under bounded perturbations of the geometry of the graph. This is a 
natural problem in connection with the theory of nilpotent and SL2 orbits in Hodge theory.

Our main result is the boundedness of variation of the ratio. For this we define the exchange graph of a given graph which encodes the exchange properties between spanning trees and spanning 2-forests in the graph. We provide a description of the connected components of this graph, and use this to prove our result on boundedness of the variations.  

\end{abstract}

\maketitle

\bigskip

\section{Introduction}

 Feynman amplitudes in quantum field theory are described as finite 
 dimensional integrals associated to graphs. 
 A Feynman graph $(G, \p)$ consists of a finite graph $G=(V,E)$, 
 with vertex and edge sets $V$ and $E$, respectively, 
 together with a collection of external momenta 
 $\underline \p=(\ps_v)_{v \in V}$, $\ps_v \in \R^D$, such that $\sum_{v\in V} \ps_v=0$. 
 Here $\R^D$ is the space-time endowed with a Minkowski bilinear form.
 
 One associates  to a Feynman graph $(G, \underline \p)$ two polynomials in the variables $\underline{Y}=(Y_e)_{e \in E}$. Denote by $\tree$ the set of all the spanning trees of the graph $G$. (Recall that a spanning tree of a connected graph is a maximal subgraph which does not contain any cycle. 
 It has precisely $|V|-1$ edges.) The first Symanzik $\psi_G$, which depends only on the graph $G$, is given by the following sum over the spanning trees of $G$:
$$
\psi_G(\underline{Y}):=\sum_{T \in \tree} \prod_{e \notin T} Y_e.
$$ 

A spanning 2-forest in a connected graph $G$ is a maximal subgraph of $G$ without any cycle and
with precisely two connected components. Such a subgraph has precisely $|V|-2$ edges. Denote by 
$\forest$ the set of all the spanning 2-forests of $G$. The second Symanzik polynomial $\phi_G$, which depends on the external momenta as well, is defined by
$$
\phi_G(\underline \p, \underline{Y}):=\sum_{F \in \forest} q(F) \prod_{e \notin F} Y_e.
$$ Here $F$ runs through the set of spanning $2$-forests of $G$, and for $F_1$ and $F_2$ the two connected components of $F$, $q(F)$ is the real number $-\langle \p_{F_1}, \p_{F_2}\rangle$, where $\p_{F_1}$ and $\p_{F_2}$ denote the total momentum entering  the two connected  components $F_1$ and $F_2$ of ${F}$, i.e., 
\[\p_{F_1} := \sum_{v\in V(F_1)} \p_{v} \qquad \qquad \p_{F_2} := \sum_{u\in V(F_2)} \p_{u}. \]

The Feynman amplitude associated to $(G, \underline \p)$ is a path integral on the space of metrics (i.e., edge lengths) on $G$ with the action given by $\phi_G/\psi_G$. It is given by 
\begin{equation*}
I_G(\underline \p) =C\int_{[0,\infty]^E}\exp(-i\,\phi_G/\psi_G)\,\, d\pi_G,
\end{equation*} for a constant $C$, and the volume form 
$d\pi_G = \psi_G^{-D/2} \prod_E dY_e$ on $\R_+^E$, c.f.~\cite[Equation (6-89)]{IZ}.

\smallskip

Motivated by the question of describing Feynman amplitudes as the infinite tension limit of bosonic 
string theory, in~\cite{ABBF} we proved results describing the ratio of the two Symanzik polynomials
in the Feynman amplitude as asymptotic behavior of the Archimedean height pairing between degree zero 
divisors in degenerating families of Riemann surfaces. A natural problem arising from~\cite{ABBF} is to consider the variation of $\phi_G/\psi_G$ obtained by perturbation of the geometry of the graph, 
in a sense that we describe below. In order the state the theorem, we need to recall the determinantal representation of the two Symanzik 
polynomials. We refer to~\cite{ABBF} where the discussion below appears in more detail.

\subsection{Determinantal representation of the Symanzik polynomials}\label{sec:mainthmintro}

Let $G = (V,E)$ be a finite connected graph on the set of vertices $V$ of size $n$ 
and with the set of edges $E = \{e_1, \dots, e_m\}$ of size $m$.  Denote by $h$ the genus of $G$, which is by definition the integer $h = m - n +1$. 

\medskip

Let $\ok$ be a ring of coefficients (that we will later assume to be either $\mathbb R $ or $\mathbb Z$), 
and consider the free $\ok$-module $\ok^E \simeq \ok^m = \bigl\{\, \sum_{i=1}^m a_ie_i\ |\ a_i \in \ok\,\bigr\}$ of rank $m$ generated by the elements of $E$. For any element $a\in \ok^E$, we denote by $a_i$ the coefficient of $e_i$ in $a$.

Any edge $e_i$ in $E$ gives a bilinear form of rank one $\langle.\,,.\rangle_i$ on $\ok^m$ by the formula
\[\langle a, b \rangle_i := a_i b_i.\]

Let $\underline{y}=\{y_i\}_{e_i\in E}$ be a collection of  elements of $\ok$ 
indexed by $E$, and consider the symmetric bilinear form 
$\alpha = \langle.\,,.\rangle_{\underline y}:= \sum_{e_i\in E}y_i\langle .\,,.\rangle_i$.
In the standard basis $\{e_i\}$ of $\ok^E$, $\alpha$ is the diagonal matrix with $y_i$ in the $i$-th entry, for $i=1, \dots, m$.  We denote by $Y : =\mathrm{diag}(y_1,\dots, y_m\}$ this diagonal matrix. 

\medskip

Let $H \subset \ok^E$ be a free $\ok$-submodule of rank $r$. 
The bilinear form $\alpha$ restricts to a bilinear form $\alpha_{|H}$ on $H$. 
Fixing a basis $B=\{\gamma_1, \dots, \gamma_r\}$ of $H$ over $\ok$, and denoting by $M$ the 
$r\times m$ matrix with row vectors $\gamma_i$ written in the standard basis $\{e_i\}$ of $\ok^E$, the restriction $\alpha_{|H}$ can be identified with the symmetric $r\times r$ matrix $MYM^\tau$ so that for two vectors $c, d \in \ok^r \simeq H$ with $a =\sum_{j=1}^r c_j \gamma_j$ and $b = \sum_{j=1}^r d_j \gamma_j$, we have 
\[\alpha(a,b) = cMYM^\tau d^\tau. \]

\medskip

The Symanzik polynomial $\psi(H,\underline y)$ associated to the free $\ok$-submodule $H\hookrightarrow \ok^E$ 
is defined as
\begin{equation*}
\psi(H,\underline y):=\det(MYM^\tau).
\end{equation*}

Note that since the coordinates of $MYM^\tau$ are linear forms in $y_1, \dots, y_m$, 
$\psi(H,\underline y)$ is a homogeneous polynomial of degree $r$ in $y_i$s.  

 For a different choice of basis $B' = \{\gamma_1',\dots, \gamma_r'\}$ of $H$ over $\ok$, the matrix $M$ is replaced by $PM$ where $P$ is the $r \times r$
  invertible matrix over $\ok$ transforming one basis into the other. So the matrix of $\alpha_{|H}$ in the new basis is given by  
  ${P}T Y T^\tau P^\tau$, , and the
  determinant gets multiplied by an element of $\ok^{\times 2}$. It follows that $\psi(H,\underline y)$ is well-defined up to an invertible element in $R^{\times 2}$. In particular, if $\ok =\mathbb Z$,  the quantity $\psi(H,\underline y)$ is independent of the choice of the basis and is therefore well-defined.

\medskip

From now on, we  fix an orientation on the edges of the graph. We have a boundary map 
$\partial: \ok^E \to \ok^V,\ e \mapsto \partial^+(e)-\partial^-(e)$, 
where $\partial^+$ and $\partial^-$ denote the head and the tail of
$e$, respectively. The homology of $G$ is defined via the exact
sequence
\begin{equation}\label{eq:6}
0 \rightarrow H_1(G,\ok) \rightarrow \ok^E \xrightarrow{\partial} \ok^V
\rightarrow \ok \rightarrow 0.
\end{equation}
The homology group $H = H_1(G,\ok)$ is  a submodule of $\ok^m$ free of rank $h$, the genus of the graph $G$, for any ring $\ok$. In particular, fixing a basis $B$ of $H_1(G,\Z)$, the polynomial 
$$\psi_G(\underline y) :=\psi(H, \underline y)$$
is independent of the choice of $B$. Writing $M$ for the $h\times m$ matrix of the basis $B$ in the standard basis $\{e_i\}$ of $\ok^E$, one sees that 
\[\psi_G(\underline y)  = \det(MYM^\tau).\] 

It follows from the Kirchhoff's matrix-tree theorem~\cite{Kir} that 
\begin{equation*}
\psi_G (\underline \ovar) = \sum_{T\in \tree}\prod_{e\not\in T}\ovar_e,
\end{equation*}
which is the form of the first Symanzik polynomial given at the beginning of this section.

\smallskip

The exact sequence \eqref{eq:6} yields an isomorphism
\begin{equation*}
\ok^{E}/H \simeq \ok^{V,0}, 
\end{equation*}
where $\ok^{V, 0}$ consists of those $x \in \ok^V$ whose coordinate sum to zero. 

Let now $\p \in \ok^{V,0}$ be a non-zero element, and let $\omega$ be any element in $\partial^{-1}(\p)$. Denote by $H_\omega = \partial^{-1} (\ok . \p) = H + \ok.\omega$, and note that $H_\omega$ is a free $\ok$-module of rank $h+1$ which comes with the basis $B_\omega = B \sqcup \{\omega\}$.

 The second Symanzik polynomial of $(G, \underline \p)$ is 
\begin{equation*}
\phi_G(\underline \p, \underline y)=\psi(H_\omega, \underline y)
\end{equation*} for the element $\omega \in \ok^E$ with $\p=\partial(\omega)$. 
The polynomial 
$\phi_G(\underline \p, \underline y)$ is homogeneous of degree $h+1$ in $y_i$'s, which is
as noted in~\cite{ABBF},  
independent of the choice of the element $\omega \in \partial^{-1}(\p)$. Writing $N$ for 
the $(h+1)\times m$ matrix for the the basis $B_\omega$ in the standard basis of $\ok^E$, we see that 
\[\phi_G(\underline \p, \underline y) = \det(NYN^\tau).\] 

The definition can be extended to $\p\in \R^D$ using the Minkowski bilinear form on $\R^D$, as discussed in~\cite{ABBF}. 

We have the following expression for the second Symanzik polynomial, c.f. e.g. to~\cite{Chaiken}, or Section~\ref{sec:proof},
\[\phi_G(\underline \p, \underline y) = \sum_{F \in \forest} q(F)\prod_{e\notin E(F)} y_{e}\,,\]
which is the form of the second Symanzik polynomial given previously. 

\subsection{Statement of the main theorem}\label{sec:mainthmintro}

Let $U$ be a topological space and $y_1,\dots, y_m : U \rightarrow \mathbb R_{>0}$ be $m$ continuous functions. Let $\p  \in (\mathbb R)^{V,0}$ be a fixed vector, and let $\psi_G(\underline y): U \rightarrow \R_{>0}$ and $\phi_G(\p,\underline y) : U \rightarrow \R_{>0}$ be the real-valued functions  on $U$ defined by the first and second Symanzik polynomials. 

\medskip

\noindent \textbf{Notation.}  We will use the following terminology in what follows: for two real-valued functions $F_1$ and $F_2$ defined on $U$, we write $F_1 =  O_{\underline y}(F_2)$ if there exist constants $c,C>0$ such that $|F_1(s)| \leq c |F_2(s)|$ on all points $s\in U$  which verify $y_1(s), \dots, y_m(s) \geq C$. 

\medskip

 Let $A: U \rightarrow \mathrm{Mat}_{m\times m}(\R)$ be a matrix-valued map taking at $s\in U$ the value $A(s)$.  Assume that $A$ verifies the following two properties
 
 \begin{itemize}
 \item[(i)] $A$ is a bounded function, i.e., all the entries $A_{i,j}$ of $A$ take values in a bounded interval $[-C, C]$ of $\R$, for some positive constant $C>0$.
 \item[(ii)] The two matrices $M(Y+A)M^\tau$ and $N (Y+A) N^\tau$ are invertible. 
 \end{itemize}
 
 One might view the contribution of $A$ as a perturbation of the standard scalar product on the edges of the graph given by 
 the (length) functions $y_1, \dots, y_m$, which can be further regarded as changing the geometry of the 
 graph, seen as a discrete metric space.  The main result of this paper is the following.
 
 \begin{thm}\label{claim:aux} Assume $A: U \rightarrow  \mathrm{Mat}_{m\times m}(\R)$ verifies the condition (i) and $(ii) $ above.  
    The difference $\frac{\det(N (Y + A)N^\tau)}{\det (M(Y+A)M^\tau)} - \frac{\det(N Y N^\tau)}{\det(MYM^\tau)}$ is  $O_{\underline y} (1)$.
 \end{thm}
 
 To prove this theorem, using Cauchy-Binet formula and some elementary observations, we are led to 
 consider a graph  which encodes the exchange properties between the spanning trees and 2-forests in the graph that we call the exchange graph of $G$,  see Definition~\ref{def:exchange}.  As our first result, we give in Theorem~\ref{thm:conn} 
a classification of the connected components of this graph. This classification theorem combined with further combinatorial arguments are then used in Section~\ref{sec:proof} to prove Theorem~\ref{claim:aux}. 

\medskip

 We note that a similar result has been proved using different tools in a recent paper of Burgos, de Jong and Holmes~\cite{BHJ} in the setting of what is called 
 \emph{normlike functions}. The perturbations in~\cite{BHJ} are required to be symmetric for the method to work, though, strictly speaking, the result in~\cite{BHJ} is more general and goes beyond the case of graphs. In comparison, the methods in this paper are purely combinatorial and  the results on the exchange graph might be of independent interest.

\medskip

 We now explain an application
 of Theorem~\ref{claim:aux} from~\cite{ABBF}, c.f. Theorem~\ref{thm:main1} below, discussed in more detail in 
 Section~\ref{sec:app}.
\subsection{Boundedness of variation of the Archimedean height pairing}
Let $C_0$ be a stable curve of genus $g$ over $\mathbb C$, and with dual graph $G=(V,E)$ which has genus $h = |E|-|V|+1$, $h\leq g$. 

Consider the versal analytic deformation $\pi : \mathcal{C} \to S$ of
$C_0$, where $S$ is a polydisc of dimension $3g-3$. The total space $\mathcal{C}$ is regular and
we let $D_e \subset S$ denote the divisor
parametrizing those deformations in which the point associated to $e$ remains singular. The divisor $D=\bigcup_{e \in E} D_e$ is a normal crossings
divisor whose complement $U=S \setminus D$ can be identified with
$(\Delta^\ast)^E \times \Delta^{3g-3-|E|}$. Assume that two collections of sections of $\pi$ are given, which we denote by 
$\sigma_1=(\sigma_{\ell, 1})_{\ell=1, \ldots, n}$ and $\sigma_2=(\sigma_{\ell, 2})_{\ell=1, \ldots, n}$. Since $\mathcal{C}$ is regular, the
points $\sigma_{\rl,i}(0)$ lie on the smooth locus of $C_0$.  Consider two fixed
 vectors
$\underline \p_1=(\ps_{\rl,1})_{\rl=1}^n$ and
$\underline \p_2=(\ps_{\rl,2})_{\rl=1}^n$ with $\ps_{\rl,i} \in \R^D$
which each satisfy the conservation of
momentum. 
We obtain a pair of relative degree zero $\R^D$-valued divisors
$$
\mathfrak A_s = \sum_{\rl=1}^n \p_{\rl, 1} \sigma_{\rl, 1}, \qquad \mathfrak
B_s = \sum_{\rl=1}^n \p_{\rl,2} \sigma_{\rl,2}.
$$

\smallskip

Assume further that $\sigma_1$ and $\sigma_2$
are disjoint on each fiber of $\pi$. 
To any pair $\mathfrak A$, $\mathfrak B$ of degree zero
(integer-valued) divisors with disjoint support on a smooth projective
complex curve $C$, one associates a real number, the archimedean height
\begin{equation*}
\langle \mathfrak A, \mathfrak B \rangle=\mathrm{Re}(\int_{\gamma_{\mathfrak{B}}} \omega_{\mathfrak A}), 
\end{equation*} by integrating a canonical logarithmic differential
$\omega_{\mathfrak{A}}$ with residue $\mathfrak{A}$ along any
$1$-chain $\gamma_{\mathfrak{B}}$ supported on $C \setminus
|\mathfrak{A}|$ and having boundary $\mathfrak B$. Coupling with the
Minkowski bilinear form on $\R^D$, the definition extends to
$\R^D$-valued divisors~\cite{ABBF}. We thus get a real-valued function 
\begin{equation*}
s \mapsto \langle \mathfrak A_s, \mathfrak B_s \rangle,
\end{equation*}  
defined on $U$.

\smallskip
For any point $s\in U$, and an edge $e\in E$, we denote by $s_e\in \Delta^*$ the $e$-th coordinate of $s$ when $U$ is identified with $\Delta^{*,E}\times \Delta^{3g-3-|E|}$.
For any point $s\in U$ and an edge $e\in E$, define $y_{e}:=\frac{-1}{2\pi }\log|s_{e}|$ and put
$\underline y = \underline y(s)=(y_{e})_{e\in E}$.  We have shown in~\cite{ABBF} that after shrinking $U$, if necessary, 
the asymptotic of the height pairing is given by the following theorem. Here 
$\phi_G(\underline \p, \underline \p', \underline{Y})$ denotes the  bilinear form associated to $\phi_G$ (which is a  quadratic form in $\underline \p$).

\smallskip

\begin{thm}[Amini, Bloch, Burgos, Fres\'an~\cite{ABBF}] \label{thm:main1} Notations as above, there exists a bounded function $h\colon U\
  \to \R$ such that 
  \begin{equation*}
    \langle \mathfrak A_{s},\mathfrak B_{s}\rangle=
    2\pi \frac{\phi_G(\underline {\mathbf p}_1^G, 
  \underline {\mathbf p}_2^G,\underline y)}{\psi_G(\underline y)}
+h(s). 
  \end{equation*}
\end{thm}

\smallskip
In Section~\ref{sec:app}, we will show how to deduce this theorem from Theorem~\ref{claim:aux} and the explicit formula obtained in~\cite{ABBF}  by means of
the nilpotent orbit theorem in Hodge theory for the variation of the archimedean height pairing, c.f. Proposition~\ref{prop:2}.

\medskip

{\bf Acknowledgments.} It is a pleasure to thank S. Bloch, J. Burgos Gil, and J. Fres\'an for the collaboration and the discussions that are the motivation behind  the results of this paper.

\section{Exchange graph}\label{sec:exchange}

Let $G=(V,E)$ be a connected multigraph with vertex set $V$ and edge set $E$. 
By a spanning subgraph of $G$ we mean a subgraph $H$ of $G$ with $V(H)=V$.  For an integer $k\geq 1$, 
a spanning $k$-forest in $G$ is a subgraph of $G$ with vertex set $V$
without any cycle which has precisely $k$-connected components; a spanning $k$-forest has precisely 
$|V|-k$ edges. For $k=1$, a spanning 1-forest is precisely a spanning tree of $G$. 
We are particularly interested in the ``exchange properties`` between spanning 2-forest and 
spanning trees in a graph $G$. 
To make this precise, we will define a new graph  $\mathscr H$ that we call the exchange graph of $G$. First we need to define an equivalence relation on the set of spanning 2-forests of $G$.

 \begin{defn}\rm
\begin{itemize}
\item For a spanning 2-forest $F$ of a graph $G$, we denote by $\mathcal P(F)=\{X, Y\}$ 
the partition $V = X \sqcup Y$ of the vertices into the vertex sets $X$ and $Y$ of the two connected components of $F$.

\item For any partition $\mathcal P$ of $V$, we denote by $E(\mathcal P)$ the set of all edges in 
$G$ which connect two vertices lying in two different elements  of $\mathcal P$.
\item  Two 2-forests  $F$ and $F'$ are called \emph{(vertex) equivalent}, and we write $F \sim_v F'$, if 
$\mathcal P(F) = \mathcal P(F')$.
\end{itemize}
 \end{defn}

The following proposition is straightforward.
\begin{prop} The following statements are equivalent for $F, F' \in \forest$:
\begin{enumerate}
\item  $F$ and $F'$ are not (vertex) equivalent.
\item there exists an edge $e\in F'$ such that 
  $F \cup\{e\}$ is a tree.
\end{enumerate}
\end{prop}

\noindent \textbf{Notation.} In what follows, for a spanning subgraph $G'$ of $G=(V,E)$ and $e\in E \setminus E(G')$, we simply write $G'+e$ to denote the spanning subgraph of $G$ with the edge set 
$E(G') \cup\{e\}$. For an edge $e\in E(G')$, we write $G'-e$ for the spanning subgraph of $G$ with the edge set $E(G') \setminus\{e\}$.

\medskip 

\begin{defn}\label{def:exchange}
 The \emph{exchange graph} $\mathscr H = \mathscr H_G = (\mathscr V, \mathscr E)$ of $G$ is defined as follows. The vertex set  
 $\mathscr V$ of $\mathscr H$ is the disjoint union of two sets $\mathscr V_{1}$ and $\mathscr V_{2}$, 
 where 
$$\mathscr V_{1} := \Bigl\{\, (F, T) \,\big|\,\, F \in \forest(G), T\in \tree(G)\,,\, E(F) \cap E(T) = \emptyset\,\Bigr\},$$ and 
$$\mathscr V_{2} := \Bigl\{\,(T, F) \,\big|\,\, T \in \tree(G), F\in \forest(G)\,,\,E(F) \cap E(T) = \emptyset\,\Bigr\}.$$
There is an edge in $\mathscr E$ connecting $(F, T)\in \mathscr V_{1}$ to $(T', F')\in\mathscr V_{2}$  
if there is an edge $e\in E(T)$ such that $F' = T -e$ and $T'  =  F + e$.
\end{defn}

\begin{defn}
 If $(T,F)$ and $(F',T')$ are adjacent in $\mathscr H$ and $F'=T -e$, we say $(F',T')$ is obtained from 
 $(T,F)$ by \emph{pivoting} involving the edge $e$.
\end{defn}

Our aim in this section is to
describe the connected components of $\mathscr H$. 

\medskip

First note that there is no isolated vertex in $\mathscr H$: consider a spanning tree $T$ and a spanning 2-forest 
$F$ of $G$ with disjoint sets of edges. Let $\mathcal P(F) = \{X,Y\}$, be the vertex sets of the 
two connected components of $F$. By connectivity of $T$, there is an edge $e$ of $T$ which joins a 
vertex of $X$ to a vertex of $Y$. It follows that $T' = F+e$ and $F'= T-e$ are spanning tree and 
2-forest in $G$, respectively, and $(F,T)\in \mathscr V_1$ is connected to 
$(T',F') \in \mathscr V_2$.

\medskip

Let now $\mathscr H_0 = (\mathscr V_0, \mathscr E_0)$ be a connected component of $\mathscr H$. Write 
$\mathscr V_0 = \mathscr V_{0,1} \sqcup \mathscr V_{0,2}$ with $\mathscr V_{0,i} \subset \mathscr V_i$, for $i=1,2$.
Note that both $\mathscr V_{0,1}$ and $\mathscr V_{0,1}$ are non-empty. 
Let $(F,T) \in \mathscr V_{0,i}$. Let $G_0=(V,E_0)$ be the spanning subgraph of $G$ having the edge set
$E_0 = E(T) \cup E(F)$. By definition of the edges in $\mathscr H$, and connectivity of $\mathscr H_0$, we have for all 
$(A, B) \in \mathscr V_0$, $E(A)\cup E(B) = E(G_0)$. We refer to $G_0$ as the spanning subgraph of $G$ associated to the connected component 
$\mathscr H_0$ of $\mathscr H$.

\medskip

\noindent{\textbf{Notation}}. For a subset $X\subset V$ of the vertices of a (multi)graph $G=(V,E)$, we denote by $G[X]$ the induced graph on $X$: it has vertex set $X$ and edge set all the edge of $E$ with both end-points lying both in $X$.

\medskip

\begin{defn}
\rm Let $X\subset V$ be a subset of vertices of $G_0$. We say $X$ is \emph{saturated} 
with respect to $G_0$ if the induced subgraph $G_0[X]$ has precisely $2|X|-2$ edges.

\noindent
A \emph{saturated component} $X$ of $G_0$ is a maximal subset of $G$ for inclusion which is saturated with respect to $G_0$.  
\end{defn}

Let $\mathscr H_0$ be a connected component of $\mathscr H$ with associated 
spanning subgraph $G_0$. 
\begin{lem} Let $X$ be a saturated subset of $G_0$. Then
For all vertices 
$(A, B) \in\mathscr V_0$, $X$ is connected in both $A$ and $B$, i.e.,  the induced graphs 
$A[X]$ and $B[X]$ are disjoint 
trees on the vertex set $X$. 
\end{lem}
\begin{proof}
 Both $A[X]$ and $B[X]$ are free of cycles. Since $G_0[X]$ has precisely $2|X|-2$ edges, and $A[X]$ and $B[X]$ are disjoint, 
 both $A[X]$ and $B[X]$ are trees on vertex set $X$.
\end{proof}

For any edge $e$ of $A$ which lie in $X$, $B+e$ has a cycle. 
Similarly, for any edge $e$ of $B$ which lie in $X$, $A+e$ has a cycle. It follows that pivoting in $G_0$ do not involve any edge in 
$X$, and so, by connectivity of $\mathscr H_0$, for any pair $(A',B') \in \mathscr V_0$, we have 
$A'[X] = A[X]$ and $B'[X] = B[X]$. 
\begin{lem}
 For two different saturated components $X$ and $X'$ of $G_0$, we have  $X \cap X' = \emptyset$.
\end{lem}
\begin{proof}
 If $X \cap X' \neq \emptyset$, the set $X \cup X'$ is connected  in  $G_0$. 
 By maximality of $X$ and $X'$, this implies, 
 $X=X'$ which is not possible since $X \neq X'$.
\end{proof}

As a corollary, the saturated components  $X_1, \dots, X_r$ of $G_0$ form a partition of $V$.
In addition, there exist for any $j=1,\dots, r$, two disjoint 
trees $T_{j,1}$ and $T_{j,2}$ with vertex set $X_j$ so that 
for any pair $(A, B) \in \mathscr V_0$, we have $A[X_j] = T_{j,1}$ and $B[X_j] = T_{j,2}$.

\medskip

 We now give another characterization of the saturated components of $G_0$ in terms of
 the connected component $\mathscr H_0$ of $\mathscr H$. 

 Define two equivalence relations 
 $\simeq_1$ and $\simeq_2$ on the set of vertices $V$ of $G_0$ as follows. For two vertices $u, v \in V$,
\begin{itemize}
  \item  we say $u \simeq_1 v$ if for any $(F,T)\in \mathscr V_{0,1}$, both vertices $u$ and $v$ 
 lie in the same connected component of $T \setminus E(\mathcal P(F))$. 
  \end{itemize}
  Similarly, 
  \begin{itemize}
\item we say $u \simeq_2 v$ if for any $(T, F)\in \mathscr V_{0,2}$, 
both vertices $u$ and $v$ lie in the same connected component of $T \setminus E(\mathcal P(F))$. 
 \end{itemize}

\begin{lem}\label{lem:a1}
Let $F$ be a spanning 2-forest in $G_0$. Let $T$ be a spanning tree of $G_0$. Suppose two vertices 
$u, v\in V$ are in two different connected components of  $T \setminus E(\mathcal P(F))$. 
There exists an edge $e\in E(\mathcal P(F)) \,\cap\, E(T)$ such that $u$ and $v$ are not connected in 
$T-e$.
\end{lem}
\begin{proof}
Denote by $S_u$ and $S_v$ the two connected components of $T \setminus E(\mathcal P(F))$ which contain $u$ and $v$, respectively. There is a path joining $S_u$ to $S_v$ in $T$. Since $S_u \neq S_v$, it contains an edge $e \in E(\mathcal P(F))$. For such an edge $e$, $u$ and $v$ are not connected in $T-e$. 
\end{proof}
The previous lemma allows to prove the following claim.
\begin{claim}\label{claim:a2}
 The two equivalence relations $\simeq_1$ and $\simeq_2$ are the same. 
\end{claim}
\begin{proof} Suppose for the sake of a 
contradiction that this is not the case.  By symmetry, let $u,v\in V$ be two vertices with  
$u\simeq_1 v$ but $u\not\simeq_2 v$. This implies the existence of 
$(T, F) \in \mathscr  V_{0,2}$ such that $u,v$ belong to two different connected components of 
$T \setminus E(\mathcal P(F))$. Applying the previous lemma, we infer the existence of an edge 
$e\in E(T)\,\cap\, E(\mathcal P(F))$ such that $u$ and $v$ are not connected in $T-e$. 
Pivoting involving $e$ gives a pair $(F', T')$ such that $u$ and $v$ lie in two different connected 
components of $F'$. In particular, it follows that $u \not\simeq_1 v$, which is a contradiction.  
This proves the claim.
\end{proof}
We denote by $\simeq$ the equivalence relation on vertices induced by $\simeq_i$.
We have actually proved the following
\begin{prop}\label{rem:uf1}\rm W properties are equivalent for $u,v\in V$:

 \noindent (1) we have $u\not \simeq v$.
 
 \noindent (2) there exists $(F, T) \in  \mathscr V_{0,1}$ 
such that $u$ and $v$ lie in different connected components of $F$.

\noindent (3) there exists 
$(T', F') \in \mathscr V_{0,2}$ such that $u,v$ lie in two different connected components of $F'$.  
\end{prop}

 Denote by $\mathcal P_{\simeq}$ the partition of $V$ induced by the equivalence classes of $\simeq$. 
 
 \begin{prop}\label{prop:cle}
  The partition $\mathcal P_{\simeq}$ coincides with the partition of $V$ into saturated components of $G_0$.
 \end{prop}

 \begin{proof} 
Any two vertices in a saturated component of $G_0$ are clearly equivalent. Thus, in order to prove the proposition, it will be enough to show that each element in $\mathcal P_{\simeq}$ saturated with respect to $G_0$. Let $X \subset V$ be an element of $\mathcal P_{\simeq}$, and consider two vertices $a, b\in X$.  Let $(T_0,F_0)\in \mathscr V_0$ be  a vertex of $\mathscr H_0$, and let $P$ be the unique 
path in $T_0$ joining $a$ and $b$. We claim that $P$ is contained in $X$. To see this, first note that there is no edge $e\in E(\mathcal P(F))$ in the path $P$: otherwise, $a$ and $b$ would lie in two different connected components of the 2-forest $T_0-e$, which is not possible by Proposition~\ref{rem:uf1}. By definition of the edges in $\mathscr H$, and by connectivity of $\mathscr H_0$, this shows that for any $(T, F) \in \mathscr V_0$, we have $P$ is included in $T$. By the definition of the equivalence relation $\simeq$, we infer that $X$ contains the path $P$. This shows that $T_0[X]$ is connected. Similarly, the induced graph $F_0[X]$ is connected. Since $E(F_0) \cap E(T_0) =\emptyset$, we infer that $X$ is a saturated set with respect to $G_0$.
   \end{proof}

We can now state the main result of this section.
\begin{thm}\label{thm:conn} 
Let $G$ be a multigraph.
\begin{itemize}
\item[(1)] The graph $\mathscr H$ is connected if and only if the following two conditions hold:
\begin{itemize}
\item[(i)] the edge set of $G$ can be partitioned as $E(G) = E(T) \sqcup E(F)$ for 
a spanning tree $T$ and a spanning 2-forest $F$ of $G$.
\item[(ii)] any non-empty subset $X$ of $V$ saturated with respect to $G_0$ consists of a single vertex. 
\end{itemize}

 \item[(2)] More generally, there is a bijection between the connected components $\mathscr H_0$ of $\mathscr H$ 
 and the pair $(G_0; \{T_{1,1},T_{1,2}, \dots, T_{r,1}, T_{r,2}\})$ where
 \begin{itemize}
  \item[(i)] $G_0$ is a spanning subgraph of $G$ which is 
  a disjoint union of a spanning tree $T$ and a spanning forest $F$ of $G$.
  \item[(ii)] denoting the maximal subsets of $V$ saturated with respect to $G_0$ by 
  $X_1, \dots, X_r$, then $T_{j,1}$ and $T_{j_2}$ are two disjoint spanning trees on the vertex set $X_j$, 
  and $E(G_0[X]) = E(T_{j,1}) \sqcup E(T_{j,2})$, for $j=1,\dots, r$.
 \end{itemize}
 Under this correspondence, the vertex set of $\mathscr H_0$ consists of all the vertices $(A,B) \in \mathscr V$ which verify
 $E(A) \cup E(B) = E(G_0)$, and for all $j=1,\dots, r$, $A[X_j] =T_{j,1} $ and $B[X_j] = T_{j,2}$.
\end{itemize}
\end{thm}

Before giving the proof of this theorem, we make the following remark.
\begin{remark}  Let $G$ be a graph whose edge set is a disjoint union of the edges of a spanning tree 
and a spanning 2-forest, and with the property that there is no saturated subset of size larger than two. 
The graph $G$ might contain spanning trees $T$ with the property that $G \setminus E(T)$ is not a spanning 2-forest. In a sense, Theorem~\ref{thm:main1} 
concerns smaller number of spanning trees of $G$, and the theorem does not seem to follow from 
the well-known connectivity of 
edge-exchanges for spanning trees.   
\end{remark}

\begin{figure}[h]
\includegraphics[width=7cm]{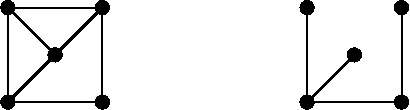}
\caption{Example of a graph $G$, on the left, which is a disjoint union of a spanning tree and a spanning 
2-forest, in which all saturated components are singletons. Note that $G$ contains a 
spanning tree $T$, given on the right,  with a complement which is not a spanning 2-forest.}
\end{figure}

The rest of this section is devoted to the proof of this theorem.

\medskip

To prove part $(1)$ of the theorem,  suppose $\mathscr H$ is connected (and so non-empty). Then (i) obviously holds. 
To prove (ii), let $X_1, \dots, X_r$ be all the different maximal subsets of vertices which are 
saturated with respect to $V$, and assume for the sake of a contradiction, and without loss of generality that 
$|X_1|>1$. Let $T_{j,1}, T_{j,2}$ be the trees on vertex set $X_j$ associated to $\mathcal H$. 
For any $(A, B) \in \mathcal V$ let $(A',B')$ be defined by 
$A' = A  - E(T_{1,1})+ E(T_{1,2})$ and $B' = B  - E(T_{1,2})+ E(T_{1,1})$. Since pivoting only involves edges which are neither in 
$T_{1,1}$ nor in $T_{1,2}$, this shows that $(A',B')$ is not a vertex of $\mathscr V$. This is a contradiction, since $A'$ and $B'$ have the same number of edges as 
$A$ and $B$, respectively, both are without cycles, and $E(G_0) = E(A') \cup E(B')$.

\medskip

We now prove the other direction. Suppose both (i) and (ii) in (1) hold.  Since any vertex $(F, T)$ in $\mathscr V_{1}$ is connected to a vertex of 
$\mathscr V_{2}$, it will be enough to prove that any two vertices 
$(T, F), (T', F') \in \mathscr V_{2}$ are connected by a path in $\mathscr H$.

We prove this proceeding by induction on the integer number 
$$r=\mathrm{diff}(T, T'):=\bigl|E(T) \setminus E(T') \bigr|.$$

\noindent $\bullet$ If  $r  = 0$, then $T = T'$, and  the claim trivially holds. 

\noindent $\bullet$ Assuming the assertion for $r$, we prove it for $r+1$. 
So let $\mathfrak v =(T, F)$, $\mathfrak v'=(T', F') \in \mathscr V_{2}$ be two vertices  with $\bigl|\,E(T) \setminus E(T') \, \bigr| = r+1$. 
For the sake of a contradiction, assume that $\mathfrak v$ and $\mathfrak v'$  (Note that this is )
are not connected in 
$\mathscr H$. Denote by $\mathscr H_0$ the connected component of $\mathscr H$ which contains $\mathfrak v$.

\medskip

We claim
 \begin{itemize}
 \item[(I)] \emph{There is no edge $e$ in $E(T) \setminus E(T')$ with $F+e \in \tree(G)$. 
 Similarly, there is no  edge $e$ in $E(T') \setminus E(T)$ with $F + e \in \tree(G)$.}
\end{itemize}

Otherwise, suppose $e \in E(T) \setminus E(T')$ is such that $F+e$ is a spanning tree of $G$. 
There exists $e' \in E(T ' ) \setminus E(T)$ such that 
$T'' = T - e + e'$ is a spanning tree of $G$. This follows from the exchange properties for the spanning trees of $G$. 
(Spanning trees of $G$ form the basis of the graphic matroid on the ground set $E$.) 
The complement $T''$ in $G$ is $F'':= F + e - e$. Since  $F +e$ is 
a spanning tree of $G$, and $e'\in F$, the subgraph $F''$ is a spanning 2-forest of $G$, and thus 
$\mathfrak v''=(T'', F'') \in \mathscr V_{2}$.  
By definition, $(T, F)$ and 
$(F + e, T - e)$  are adjacent in $\mathscr H$. Moreover, 
$(F +e, T-e)$ and $\mathfrak v''$ are adjacent in $\mathscr H$. 
Since $\mathrm{diff}(T'', T') =r$,  by hypothesis of the induction, $(T'', F'')$ and 
$(T',F')$ are connected by a path in 
$\mathscr H$. Thus   $(T, F)$ and $(T',F')$ are connected, which is a contradiction to the
assumption we made. This proves our claim (I).  

\medskip

As a consequence of (I) we now show that
\begin{itemize}
 \item[(II)] \emph{We have $F \sim_v F'$, i.e., the two partitions $\mathcal P(F)$ and $\mathcal P(F')$ 
 of $V$ coincide.}
\end{itemize}

\medskip

Let $\mathcal P(F) = \{X, Y\}$ and $\mathcal P(F') = \{X', Y'\}$, and 
suppose for the sake of a contradiction that the two partitions are not equal. The partition 
$\mathcal P(F)$ (resp. $\mathcal P(F')$) induces a partition of both $X'$ and $Y'$ (resp. $X$ and $Y$). 
One of these four induced partitions has to be non-trivial: by this we mean that, 
without loss of generality, 
we can assume that
 $Z:=X \cap X'$ and $W:=X \cap Y'$ are both non-empty. Since $F[X]$ is connected, there is an edge $e=\{u,v\} \in F$ with $u\in Z$ and $v\in W$. 
This edge does not belong to 
$F'$ since it joins a vertex in $X'$ to a vertex in $Y'$. Therefore,  $e\in T'$, and 
thus $e\in E(T')\setminus E(F)$. Moreover, $F' +e$ is a spanning tree, which is a contradiction to (I). This proves our claim (II).

\medskip

Let $\mathcal P(F) = \mathcal P(F') = \{X, Y\}$.  Denote by $\mathcal P_X$ the partition of $X$ given by the vertex sets of the connected 
components of $T[X]$. Also, denote by $\mathcal P'_X$ the partition of $X$ induced by the connected components of $T'[X]$. Similarly, define $\mathcal P_Y$ and $\mathcal P_Y'$. Let $E(\mathcal P_X)$ (resp. $E(\mathcal P_Y)$) 
be the set of all edges $e$ of $G$ with end-points in two different members of 
$\mathcal P_X$ (resp. $\mathcal P_Y$), 
respectively. 
Similarly, define $E(\mathcal P_X')$ and $E(\mathcal P_Y')$.

\medskip

We now claim.
\begin{itemize}
 \item[(III)] \emph{The intersections $E(T') \cap E(\mathcal P_X)$, 
 $E(T') \cap E(\mathcal P_Y)$, $E(T) \cap E(\mathcal P_X')$, $E(T) \cap E(\mathcal P_Y')$ are all empty.}
\end{itemize}

Otherwise, without loss of generality, suppose there is an edge $e' \in T'$ 
with $e'\in E(\mathcal P_X)$. Since $e'$ joins two different connected components of $T[X]$, we have 
$e'\in F$. The graph $T + e'$ has a cycle, which, once again since $e'$ joins two different connected components of $T[X]$, must include an edge $e\in E(\mathcal P(F))$. Since $\mathcal P(F) = \mathcal P(F')$, we have $e\in E(T')$. 

\noindent Let $\mathfrak v_1 = (F_1, T_1) $ with 
 $F_1 = T-e$ and $T_1 = F+e$, and 
$\mathfrak v_2 = (T_2, F_2)$ with $T_2 = F_1 +e'$ and $F_2 = T_1 - e'$.
By choices of $e$ and $e'$, both $\mathfrak v_1$ and $\mathfrak v_2$ are vertices in $\mathscr H$, 
and $\mathfrak v, \mathfrak v_1, \mathfrak v_2$ forms a path of length two.  We have 
$\mathrm{diff}(T_2, T') = \mathrm{diff}(T, T') =r+1$. 

Since $F_2$ contains $e$, the edge $e$ lies entirely in one of the two connected components of $F_2$, 
and so $\mathcal P(F_2) \neq \mathcal P(F)$. Since by our assumption, 
$\mathfrak v$ and $\mathfrak v'$ are not connected in $\mathscr H$, $\mathfrak v''$ and 
$\mathfrak v'$ are not connected in $\mathscr H$. Applying the above reasoning to $\mathfrak v''$ and 
$\mathfrak v'$, we must have by (II) that $\mathcal P(F_2) = \mathcal P(F')$. 
Since $\mathcal P(F') = \mathcal P(F)$, this gives a contradiction. This proves our claim (III).

\medskip

As an immediate corollary of (III), we get

\begin{itemize}
\item[(IV)] \emph{We have $\mathcal P_X =\mathcal P'_X$ and $\mathcal P_Y =\mathcal P'_Y$}.
\end{itemize}
\medskip

Indeed, since  $E(T')\cap E(\mathcal P_X) =\emptyset$, any subset $Z'$ of $X$ with $T'[Z']$ connected, should be entirely included in an element of $\mathcal P_X$. This in particular, when applied to
each $Z' \in \mathcal P_X'$, shows the existence of $Z \in \mathcal P_X$ with $Z' \subseteq Z$, which shows that $\mathcal P_X'$ is a refinement of $\mathcal P_X$.
 By symmetry, we get that 
 $\mathcal P_X$ is a refinement of $\mathcal P_X'$. Thus, $\mathcal P_X = \mathcal P'_X$. The equality $\mathcal P_Y = \mathcal P'_Y$ follows similarly.

\medskip

As a corollary, we get
\begin{itemize}
\item[(V)] \emph{The equality $E(\mathcal P_X) \sqcup E(\mathcal P_Y) = 
E(\mathcal P_X') \sqcup E(\mathcal P_Y')$ holds}
\end{itemize}

\medskip
To see this, note that by (III), $E(\mathcal P_X) \sqcup E(\mathcal P_Y) $ and $  
E(\mathcal P_X') \sqcup E(\mathcal P_Y')$ are subsets of $E(F)\cap E(F')$, we thus get the equality of the two sets from (IV).

\medskip

By the definition of $\mathscr H$, all the vertices $\mathfrak v_2 = (T_2, F_2)$ of 
$\mathscr H$ at distance 2 from $(T, F)$ are of the form 
$T_2 =  T -e_1 + e_2$ and $F_2 = F -e_1 + e_2$ for $e_1\in E(\mathcal P(F))$ and 
$e_2 \in E(\mathcal P_X)\sqcup E(\mathcal P_Y)$.
Indeed, $F + e_1 \in \tree(G)$ implies $e_1\in E(\mathcal P(F))$. Similarly, $T- e_1 +e_2 \in \tree(G)$ implies $e_2 \in E(\mathcal P_X)\sqcup E(\mathcal P_Y)$.

By (II), $E(P(F)) = E(\mathcal P(F'))$, and  by (V), 
$E(\mathcal P_X)\sqcup E(\mathcal P_Y) = E(\mathcal P_X')\sqcup E(\mathcal P_Y')$. Thus, for such a vertex $\mathfrak v_2$, the pair
 $\mathfrak v_2' = (T_2',F_2')$ defined by
$T_2' = T' - e_1 + e_2$ and $F_2' = F' + e_1 - e_2$ is also a vertex of $\mathscr H$ 
at distance two from $\mathfrak v'$. In addition, we have 
$\mathrm{diff}(T_2, T'_2) = \mathrm{diff}(T, T') = r+1$.

Since by our assumption, $\mathfrak v$ and $\mathfrak v'$ are not 
connected in $\mathscr H$, any pair of vertices $\mathfrak v_2$ and $\mathfrak v'_2$ 
as above are not connected in $\mathscr H$. 

Since $T\neq T'$, there is an edge $e_{\star} \in E(T')\setminus E(T)$.  For any choice of $e_1, e_2$ as 
above, we have $e_{\star}\neq e_1, e_2$, and thus we must have $e_{\star}\in E(T_{2}') \setminus E(T_2)$. 
\medskip

Applying the same reasoning to the pair $\mathfrak v_2$ and $\mathfrak v'_2$, and proceeding inductively on $k$, we infer that for any 
vertex $\mathfrak v_{2k} = (T_{2k}, F_{2k})$ of $\mathscr H$ 
obtained from $(T, F)$ by an ordered sequence of pivoting involving edges $e_1, e_2, \dots, 
e_{2k-1}, e_{2k}$,  the pair $\mathfrak v_{2k}' = (T_{2k}', F_{2k}')$ obtained from $\mathfrak v'$ 
by pivoting involving the same ordered sequence of edges $e_1, e_2, \dots, e_{2k-1}, e_{2k}$ is a 
vertex of $\mathscr H$, and we have by (I)-(V):
\begin{itemize}
 \item $\mathcal P(F_{2k}) = \mathcal P(F_{2k}') = \{X_{2k}, Y_{2k}\}$ (with $X_{2k}$ and $Y_{2k} $ depending on the sequence of edges $e_1, \dots, e_{2k}$), 
\item $E(\mathcal P_{X_{2k}}) \sqcup E(\mathcal P_{Y_{2k}})= E(\mathcal P_{X_{2k}}') \sqcup E(\mathcal P_{Y_{2k}}')$.
\item $\mathrm{diff}(\mathfrak v_{2k}, \mathfrak v_{2k}')=r+1$,  and $\mathfrak v_{2k}$ and $\mathfrak v_{2k}'$ are not connected in $\mathscr H$.
\item $e_* \in E(T_{2k}') \setminus E(T_{2k})$
\end{itemize}

To get a contradiction, note that all the vertices of $\mathscr H$ appear among the set of vertices $\mathfrak v_{2k}$, 
and we have $e_* \in E(T'_{2k}) \setminus E(T_{2k}) \subset E(F_{2k})$. It follows that the two end-points 
of $e_*$ are in the same equivalence class of $\simeq$. Since $\mathcal P_{\simeq}$ coincides with the partition of 
$V$ into saturated components of $G_0$, this leads to a contradiction to the assumption that all the saturated components are singletons. 
This final contradiction proves the step $r+1$ of our induction and finishes the proof of the first part of our theorem. 

\medskip

Part (2) follows directly from part (1):  contract all the edges lying in a saturated component in 
$G_0$ in order to get the graph $\widetilde G_0$. 
One verifies that in $\widetilde G_0$, all the saturated components are singleton, and the edges of $\widetilde G_0$ are 
a disjoint union of the edges of a spanning tree and a spanning 2-forest. Thus by part (1), the graph 
$\mathscr H_{\widetilde G_0}$ is connected. There is an isomorphism from $\mathscr H_0$ to $\mathscr H_{\widetilde G_0}$ which sends 
a pair $(A,B)$ in $\mathscr V_0$ to the pair $(\widetilde A, \widetilde B)$ in   $\mathscr H_{\widetilde G_0}$ obtained by contracting all the edges 
in the trees $T_{j,1},T_{j,2}$, for $j=1,\dots, r$.

\section{Proof of Theorem~\ref{claim:aux}} \label{sec:proof}

 For an $r\times t$ matrix $X$, and subsets $I \subset\{1,\dots, r\}$ and $J \subset \{1,\dots, t\}$ with $|I| = |J|$, we note by $X_{I,J}$ the square $|I| \times |I|$ submatrix of $X$ with row and columns in $I, J$, respectively. 

 If $r\leq t$, and $I =\{1,\dots, r\}$ and $J \subset \{1,\dots, t\} $, we simply write $X_{J}$ instead of $X_{I,J}$.  

\medskip

We use the notation of the introduction: choosing  a basis $\gamma_1, \dots, \gamma_h$ 
for $H_1(G, \mathbb Z)$, we denote by $M$ the $h\times m$ matrix of the coefficients of $\gamma_i$ 
in the standard 
basis $\{e_i\}$ of $\R^m$.  Similarly, for the element $\omega \in \R^E$ in the inverse image $\partial^{-1}(\omega)$ 
of the external momenta vector $\p = (\p_v)$, we denote by $H_\omega$ the $(h+1)$-dimensional 
vector subspace of $\R^m$ generated by $\omega$ and $H_1(G,\R)$. The space
$H_\omega$ comes with a basis consisting of $\gamma_1, \dots, \gamma_h, \omega$, and we denote by $N$ 
the $(h+1)\times m$ matrix of the coefficients of this basis in the standard 
basis $\{e_i\}$ of $\R^m$.

\medskip

 By Cauchy-Binet formula, we have 
\begin{align}\label{eq:CB1}
 \det(N Y N^\tau) &= \sum_{\substack{I,J \subset \{1,\dots, m\}\\ |I| = |J| = h+1}} \det(N_I) \det(Y_{I,J}) \det(N_J),
\end{align}
which, using that $Y$ is diagonal, can be further reduced to 
$$\det(N Y N^\tau) =\sum_{\substack{I,J \subset \{1,\dots, m\}\\ |I| = h+1}} \det(N_I)^2 y^I,$$
where as usual, we pose $y^I := \prod_{i\in I} y_i$.  Similarly, we have
 
 \begin{align}\label{eq:CB2}
 \det(M Y M^\tau) &= \sum_{\substack{I \subset \{1,\dots, m\}\\ |I| = h}} \det(M_I)^2 y^I.
\end{align}

  For a subgraph $F$ in $G$, by an abuse of the notation,  we write $F^c$ (instead of $E \setminus E(F)$) 
  for the set of edges of $G$ not in $F$.

 \begin{lem} \label{lem:aux1}\begin{itemize}
              \item[(1)] For a subset $I \subseteq \{1,\dots, m\}$ of size $h$, we have $\det(M_I) \neq 0$ if and only if $I = T^c$ for 
             a spanning tree $T$ of $G$. In this case, we have 
              $\det(M_I)^2 = 1$.	
              \item[(2)] For a subset  $I \subseteq \{1,\dots, m\}$ of size $h+1$, we have $\det(M_I) \neq 0$ if and only 
              if $I =  F^c$ for a spanning 2-forest $F$ of $G$. In this case, we have 
              $\det(N_I)^2 = q(F) = (\sum_{v\in X)} \p_v).(\sum_{v\in X}\p_v)$, where $\{X,Y\}$ denotes the partition of $V$ given by $F$.	
             \end{itemize}
 \end{lem}
 
 \begin{proof} These facts are folklore. Here we only prove (2), part (1) has a similar proof. 

 Denote by $e_{i_1}, \dots, e_{i_{h+1}}$ the $(h+1)$ edges of $I$. 
 Developing $\det(N)$ with respect to the last row (which corresponds to the coefficients of $\omega$), 
 we have  \[\det(N_I) = \sum_{j=1}^m (-1)^j \omega(e_i) \det(M_{I \setminus\{e_{i_j}\}}). \]
 From the first part, it follows that $\det(N_I) =0$ if none of $I - e_{i_j}$ is the complement set of edges 
 of a spanning tree, i.e., if $I$ is not of the form $F^c$ for a spanning 2-forest of $G$. So suppose now that 
 $I= F^c$, denote by $X, Y$ the partition of $V$ induced by $F$, and without loss of generality, let 
 $e_{i_1}, \dots, e_{i_r}$ be the set of all the edges in $E(\mathcal P(F))$. We can assume that $e_i$'s are all 
 oriented from $X$ to $Y$.  Let $T_j = F \cup\{e_{i_j}\}$ the spanning tree $F \cup\{e_{i_j}\}$ for $j=1,\dots, r$. It follows that 
  \[\det(N_I) = \sum_{j=1}^r (-1)^j \omega(e_i) \det(M_{T_j^c}). \]
 Since $\partial(\omega) = \p$, and the edges $e_{i_1},\dots, e_{i_r}$ are oriented from $X$ to $Y$, it follows that 
 \[\sum_{j=1}^r \omega(e_{i_j}) = \sum_{v\in X} \p_v. \]
 
 So the lemma follows once we prove that $(-1)^j \det(M_{I \setminus\{e_{i_j}\}})$ takes the same value for all 
 $j=1,\dots, j$. By symmetry, it will be enough to prove 
 $ \det(M_{T_{1}^c}) + \det(M_{T_2^c}) =0$. By multi-linearity of the determinant with respect to the 
 columns we see that $\det(M_{T_{1}^c}) + \det(M_{T_2^c}) = \det(P)$ where $P$ is the $h\times h$ matrix with
 the first column 
 equal to the sum of the first columns of 
 $M_{T_{1}^c}$ and $M_{T_{1}^c}$, and the $j$'th column equal to the $j$'th column of $M_{T_{1}^c}$ 
 (which is the same as that of 
 $M_{T_{1}^c}$), for $j\geq 1$. So it is enough to show that $\det(P)=0$. 
 The subgraph $F \cup\{e_{i_1}, e_{i_2}\}$ has  a cycle $\gamma$ which contains $e_{i_1}, e_{i_2}$ 
 from $F^c$ and all the other edges are in 
 $F$. Writing $\gamma$ as a linear combination $\gamma=\sum_{j=1}^h\gamma_j$ of the cycles $\gamma_j$, we show that 
 $(a_1, \dots, a_h)P=0$. The first coefficient of $(a_1, \dots, a_h)P$ is zero since the cycle 
 $\gamma$ has $e_{i_1}$ and $e_{i_2}$ with different signs. All the other coordinates of 
 $(a_1, \dots, a_h)P$ are zero since the only edges of $\gamma$ in $F^c$ are $e_{i_1}$ and $e_{i_2}$.
 \end{proof}

 \begin{remark} The proof of the above lemma shows the following useful property. Suppose that 
 $I$ and $J$ are the complement of the edges of two (vertex-)equivalent 2-forests $F_1 \sim_v F_2$ 
 inducing the partition of $V =X \sqcup Y$, and $e \in E(\mathcal \{X,Y\})$ is an edge with one end-point in each of $X$ and $Y$, so
 both  $T_1 = F_1 \cup \{e\}$ and $T_2= F_2  \cup \{e\}$ are spanning trees. Then 
$$\frac{\det(N_I)}{\det(M_{T_1^c})} = \frac{\det(N_{J})}{\det(M_{T_2^c})} = \pm\sum_{e\in E(X,Y)}\omega(e),$$ 
where $e$ in the above sume runs over all the oriented edges from $X$ to $Y$. 
In particular, we have 
\begin{equation}\label{eq:u1}
 \det(N_I) \det(N_J) = q(F_1) \det(M_{T_1^c})\det(M_{T_2^c}) =q(F_2) \det(M_{T_1^c})\det(M_{T_2^c}).
\end{equation}
 \end{remark}

%
%

\medskip

From Lemma~\ref{lem:aux1} we infer that 
in the sum~\eqref{eq:CB2} (resp.~\eqref{eq:CB2}) above 
describing $\det(M Y M^\tau)$ (resp. $\det(N Y N^\tau)$), 
the only non-zero terms correspond to subsets $I$ which are complements of the edges a spanning tree (resp. spanning 2-forest) of $G$.

\medskip
 Consider the set-up of Theorem~\ref{claim:aux} as in the introduction, where 
 $U$ is a topological space and $y_1,\dots, y_m : U \rightarrow \mathbb R_{>0}$ are 
 $m$ continuous functions. Denote by $Y$ the diagonal matrix-valued function on $U$ given by $Y(s) =\mathrm{diag}(y_1(s), \dots, y_m(s))$. 
 Let $\p  \in (\mathbb R)^{V,0}$ be a fixed vector

Define two real-valued functions $f_1$ and $f_2$ on $U$ by
\begin{align}
 f_1(s) := \det(MYM^\tau)= \sum_{\substack{T \in \tree \\ I= T^c}}  y(s)^I,\,\,\textrm{and}
\end{align}
and
\[f_2(s) := \det(N Y N^\tau) = \sum_{\substack{F\in \forest \\ I = F^c }} q(F) y(s)^I,\]
at each point $s\in U$.
Note that $f_1(s) = \phi(\underline y(s))$, for $\phi$ the first Symanzik polynomial, and 
$f_2(s) = \psi_G(\omega, \underline y(s)))$, for $\psi$ the second Symanzik polynomial of the graph $G$.

\medskip

Let now $A: U \rightarrow \mathrm{Mat}_{m\times m}(\R)$ be a matrix-valued map taking at $s\in U$ the value 
 $A(s)$.  Assume that $A$ verifies the two properties
 
 \begin{itemize}
 \item[(i)] $A$ is a bounded function, i.e., all the entries $A_{i,j}$ of $A$ take values in a bounded interval $[-C, C]$ of $\R$, for some positive constant $C>0$.
 \item[(ii)] The two matrices $M(Y+A)M^\tau$ and $N (Y+A) N^\tau$ are invertible. 
 \end{itemize}

\medskip

Define real-valued functions $g_1, g_2$ on $U$ by $g_1(s):= \det(M (Y+A) M^\tau)$ and 
$g_2(s) = \det(N (Y+A) N^\tau)$, we have by Cauchy-Binet formula, 

\begin{align*}
  g_1 = \sum_{\substack{T_1, T_2 \in \tree  \\ I = T_1^c, J = T_2^c}} \det(M_I) \det(Y+A)_{I,J} \det(M_J), \,\,\textrm{and}
 \end{align*}

\begin{align*}
 g_2 &= \sum_{\substack{ F_1, F_2\in \forest  \\ I = F_1^c, J = F_2^c}} \det(N_I) \det(Y+A)_{I,J} \det(N_J).
\end{align*}

 To prove Theorem~\ref{claim:aux}, we must show that $g_2/g_1 - f_2/f_1 = O_{\underline y}(1)$ on $U$. 
 Observe first that 

\begin{claim}\label{claim:aux1}
 There exist constants $c_1,c_2,  C>0$ such that
\begin{equation}\label{eq:o1}
 c_1 f_1(s) \,<\, g_1(s)\,<\, c_2 f_1(s),
\end{equation}
for all points $s\in U$ with $y_1(s), \dots, y_m(s) \geq C$. 
\end{claim}
\begin{proof}
By assumption, all the coordinates of $A$ are bounded functions on $U$. Developing the determinant
$\det(Y+A)_{I,J}$ as a sum (with $\pm$ sign) over permutations of the products of entries of $(Y+A)_{I,J}$, 
one observes that each term in the sum is the product of a bounded function with 
a monomial in the $y_j$'s for indices $j$ in a subset of $I\cap J$. 
 For $I \neq J$, these terms become $o(y^I)$. 
Also for $I=J$, all the terms but the unique one coming from the product of the entries on the diagonal which 
gives $y^I$ are $o(y^I)$. Since $f_1 =\sum_{T\in \tree} y^{T^c}$, the assertion follows.
\end{proof}

Therefore, in order to prove Theorem~\ref{claim:aux}, it will be enough to show that 
\begin{equation}\label{claim:fin}
 g_2f_1 - g_1 f_2 = O_{\underline y}(f_1^2).
\end{equation}

\medskip

In considering the terms in $g_2f_1 - g_1 f_2$  it will be convenient to define the bipartite graph 
$\mathfrak G = (\mathfrak V , \mathfrak E)$, a variation of the exchange graph introduced in the previous section. 
The vertex set $\mathfrak V$ of $\mathfrak G$ is partitioned into two sets 
$\mathfrak V_1$ and $\mathfrak V_2$ with 
\[\mathfrak V_1 := \Bigl\{ (F_1, F_2, T) \,|\, F_1, F_2 \in \forest, T\in \tree\Bigr\},\]
and
\[\mathfrak V_2 := \Bigl\{ (T_1, T_2, F) \,|\, T_1, T_2 \in \tree, F\in \forest\Bigr\}\]

There is an edge between $(F_1, F_2, T) \in \mathfrak V_1$ and $(T_1, T_2, F)\in \mathfrak V_2$ in $\mathfrak G$ iff there is an edge $e\in E$ such that $T = F +e $, and $F_1 = T_1 -e$ and $F_2 = T_2 -e$.

\begin{defn}\rm
If $(F_1, F_2, T) \in \mathfrak V_1$ and $(T_1, T_2, F)  \in \mathfrak V_2$ are adjacent in $\mathfrak G$, we say $(T_1, T_2, F)  \in \mathfrak V_1$ is obtained from $(F_1, F_2, T)$ 
by \emph{pivoting involving the edge} $e$ (with $E(T) \setminus E(F) =\{e\}$).
\end{defn}

Define two weight functions $\xi, \zeta: \mathfrak V \rightarrow C^0(U,\R)$ on the vertices of $\mathfrak G$ 
as follows.  For $(F_1, F_2, T) \in \mathfrak V_1$, let
\begin{align*}
\xi(F_1, F_2, T)&: =  \det(Y+A)_{F_1^c,F_2^c}\,\,y^{T^c},\\
\zeta(F_1, F_2, T) &:= \det(N_{F_1^c})\det(N_{F_2^c}) \, \xi(F_1, F_2, T),
\end{align*}
and for $(T_1, T_2, F) \in \mathfrak V_2$, define 
\begin{align*}
\xi(T_1, T_2, F) &:= \det(Y+A)_{T_1^c,T_2^c}\,\,y^{F^c},\\
\zeta(T_1, T_2, F) &:= \det(M_{T_1^c})\det(M_{T_2^c}) q(F)\,\xi(T_1, T_2, F).
\end{align*}
Note that we have 
\begin{equation} \label{eq:us1}
g_2f_1 = \sum_{(F_1, F_2, T) \in \mathfrak V_1} \zeta(F_1, F_2, T),
\end{equation}
and 
\begin{equation} \label{eq:us2}
g_1f_2 = \sum_{(T_1, T_2, F) \in \mathfrak V_2} \zeta(T_1, T_2, F).
\end{equation}

We have the following
\begin{claim}\label{claim:trivial} \begin{itemize}
               \item For any $(F_1, F_2, T) \in \mathfrak V_1$, we have 
\[\xi(F_1, F_2, T) = O_{\underline y}(y^{F_1^c\cap F_2^c} y^{T^c}).\]
\item For any $(T_1, T_2, F) \in \mathfrak V_2$, we have 
\[\xi(T_1, T_2, F) = O_{\underline y}(y^{T_1^c \cap T_2^c} y^{F^c}).\]
\item For two adjacent vertices $(F_1, F_2, T)\in \mathfrak V_1$ and $(T_1, T_2, F) \in \mathfrak V_2$, we have 
\[\xi(F_1, F_2, T)  = \xi(T_1, T_2, F) + O_{\underline y}(f_1^2).\] 
              \end{itemize}
\end{claim}
\begin{proof} The first two assertions are straightforward. To prove the last one, let $e$ be the unique edge in $T \setminus F$. We have  
\[\det(Y+A)_{F_1^c,F_2^c} = y_e \det(Y+A)_{T_1^c,T_2^c} + O_{\underline y}(y^{T_1^c}),\]
Multiplying both sides by $y^{T^c}$ gives
\[\xi(F_1, F_2, T)  = \xi(T_1, T_2, F) + O_{\underline y}(y^{T^c} y^{T_1^c}) = O_{\underline y}(f_1^2).\] 
\end{proof}

\begin{defn}\rm
 \begin{itemize}
  \item A tuple $(F_1, F_2, T) \in \mathfrak V_1$ is called \emph{special} if $F_1 \not\sim_v F_2$. 
  \item A tuple $(T_1, T_2, F) \in \mathfrak V_2$  is called \emph{special} if there exists either $e \in E(T_1) \setminus E(T_2)$ or $e\in E(T_2) \setminus E(T_1)$ such that $F + e$ is a spanning tree. 
 \end{itemize}
\end{defn}
The following observations are crucial for the proof of our theorem. They show that connected components of $\mathfrak G$ which contain special vertices have only "light weight" vertices.

\begin{claim}\label{claim:k}
\begin{enumerate}
 \item For any special vertex $\mathfrak w$ in $\mathfrak V$, we have 
  $$\xi(\mathfrak w) = O_{\underline y}( f_1^2).$$
 \item For any vertex $\mathfrak v \in \mathfrak V$ connected by a path in $\mathfrak G$ to a special vertex $\mathfrak u$, we have 
 \[\xi(\mathfrak v) = O_{\underline y}( f_1^2).\]
\end{enumerate}
\end{claim}

\begin{proof}(1) If $\mathfrak w = (F_1, F_2, T)\in \mathfrak V_1$, then since $F_1 \not\sim_v F_2$, there exists an edge $e\in F_2$ such that $T_1=F_1 + e$ is  a tree. 
In this case, we have $F_1^c\cap F_2^c  = (F_1 \cup F_2)^c\subset T_1^c$, and so we have by Claim~\ref{claim:trivial}, 
\[\xi(F_1, F_2, T) = O_{\underline y}(y^{F_1^c \cap F_2^c} y^{T^c}) = O_{\underline y}(y^{T_1^c} y^{T^c}) = O_{\underline Y}(f_1^2).\] 
Similarly, if $\mathfrak w = (T_1, T_2, F)\in \mathfrak V_2$ is special, assume without loss of generality that there is an edge $e\in E(T_1)\setminus E(T_2)$ such that $T = F+e$ is a spanning tree. 
Since $e\notin T_2$, applying Claim~\ref{claim:trivial}, 
we get
 \[\xi(T_1, T_2, F) = O_{\underline y} (y^{T_1^c \cap T_2^c} y^{F^c})  = O_{\underline y} (y^{T_2^c})(y^{T^c}) = O_{\underline y} (f_1^2).\]

\noindent (2) This follows from (1) and the third assertion in Claim~\ref{claim:trivial}.
\end{proof}

\begin{defn}\rm Let $\p \in \mathbb R^{V,0}$ be the vector of external momenta. 
For any vertex $\mathfrak u = (T_1, T_2, F) \in \mathfrak V_2$, define $q(\mathfrak u) := q(F)$. For any  vertex $\mathfrak v = (F_1, F_2, T) \in \mathfrak V_1$ with $F_1 \sim_v F_2$, define $q(\mathfrak v):=q(F_1) = q(F_2)$.
\end{defn}
We have the following useful property. 
\begin{claim}\label{claim:k2}
For two adjacent vertices $\mathfrak v \in \mathfrak V_1$ and $\mathfrak u \in \mathfrak V_2$ with $\mathfrak v$ non-special, we have 
\[q(\mathfrak u) \zeta(\mathfrak v) = q(\mathfrak v) \zeta(\mathfrak u) + O_{\underline y}(f_1^2).\]
\end{claim}
\begin{proof} 
Let $\mathfrak v=(F_1, F_2, T) \in \mathfrak V_1$ and $\mathfrak u=(T_1, T_2, F) \in \mathfrak V_2$, and let $e$ be the edge in $E$ with
$T = F +e$,  $T_1 = F_1 + e$ and $T_2 = F_2 + e$. 
By assumption, we have $F_1 \sim_v F_2$. We already noted that 
\[\xi(\mathfrak v) = \xi(\mathfrak u) + O_{\underline y}(f_1^2).\]
Multiplying both sides of this equation by $\det(N_{F_1^c})\det(N_{F_2^c})q(F)$, 
and using Equation~\eqref{eq:u1}, $\det(N_{F_1^c})\det(N_{F_2^c}) =\det(M_{T_1^c})\det(M_{T_2^c}) q(F_1),$
gives the result. 
\end{proof}

As immediate corollary of the above claims, we get
\begin{cor}\label{cor:hf} \begin{itemize}
\item Let $\mathcal G$ be a connected component of $\mathfrak G$. 
If $\mathcal G$ contains a special vertex, then for any vertex $\mathfrak v\in \mathfrak V(\mathcal G)$, 
we have 
 \[\zeta(\mathfrak v) = O_{\underline y}(f_1^2).\]
 \item Let $\mathcal G$ be a  connected component of $\mathfrak G$ which does not contain any special 
 vertex.  There exists $\zeta$ such that for any vertex $\mathfrak w$ of $\mathcal G$, we have 
 \[\zeta(\mathfrak w) = q(\mathfrak w) \zeta + O_{\underline y}(f_1^2). \]
 \end{itemize}
\end{cor}
\begin{proof} The first assertion follows from Claim~\ref{claim:k}. 
To prove the second part, let $\mathfrak v$ be a vertex of $\mathcal G$, and choose 
$\zeta$ so that $\zeta(\mathfrak v) = q(\mathfrak v) \zeta$. The assertion now 
follows from Claim~\ref{claim:k2} and the connectivity of $\mathcal G$. 
\end{proof}

The following proposition finally allows us to prove Theorem~\ref{claim:aux}.
\begin{prop}\label{claim:final}
 Let $\mathcal G = (\mathcal V, \mathcal E)$ be a connected component of $\mathfrak G$ with vertex set 
 $\mathcal V = \mathcal V_1 \sqcup \mathcal V_2$, with $\mathcal V_i = \mathcal V \cap \mathfrak V_i$. 
 Suppose  that 
 $\mathcal G$ does not contain any special vertex. Then we have 
 \[\sum_{\mathfrak u\in \mathcal V_1}q(\mathfrak u)  =\sum_{\mathfrak w\in \mathcal V_2} q(\mathfrak w).\]
\end{prop}

We will give the proof of this proposition in the next section. Let us first explain how to deduce Theorem~\ref{claim:aux} assuming this result.
\begin{proof}[Proof of Theorem~\ref{claim:aux}] 
We have to show that $g_2f_1 - g_1 f_2 = O_{\underline y}(f_1^2)$. 
 Let $\mathcal G_1 = (\mathcal V_1, \mathcal E_1), \dots, \mathcal G_N=(\mathcal V_N, \mathcal E_N)$ be all the connected components of $\mathfrak G$. \
 For each $i=1,\dots, N$,  denote by $\mathcal V_{i,1}$ $\mathcal V_{i,2}$ the intersection of 
 $\mathcal V_i$ with $\mathfrak V_1$ and $\mathfrak V_2$ respectively. 
 Using Equations~\eqref{eq:us1} and~\eqref{eq:us2}, we can write 
 \begin{align*}
  g_2f_1 - g_1 f_2 &= \sum_{\mathfrak v\in \mathfrak V_1} \zeta(\mathfrak v)- \sum_{\mathfrak u\in \mathfrak V_2}\zeta(\mathfrak u)\\
  &=\sum_{i=1}^N \Bigl(\sum_{\mathfrak v\in \mathcal V_{i,1}} \zeta(\mathfrak v) - \sum_{\mathfrak u\in 
  \mathcal V_{i,2}} \zeta(\mathfrak u)\Bigr).  
 \end{align*}
 For each $1\leq i\leq N$, we have the following two possibilities. Either, $\mathcal G_i$ 
 contains a special vertex, in which case we have $\zeta(\mathfrak w) =O_{\underline y}(f_1^2)$ 
 for all $\mathfrak w\in \mathcal V(\mathcal G_i)$. In particular,
 \[ \sum_{\mathfrak v\in \mathcal V_{i,1}} \zeta(\mathfrak v) - \sum_{\mathfrak u\in 
 \mathcal V_{i,2} }\zeta(\mathfrak u) = O_{\underline y}(f_1^2).\]
 Or $\mathcal G_i$ does not contain any special vertex, in which case, applying
  Corollary~\ref{cor:hf} and Proposition~\ref{claim:final}, we must have 
 \begin{align*}
 \sum_{\mathfrak v\in \mathcal V_{i,1}} \zeta(\mathfrak v) - \sum_{\mathfrak u\in \mathcal V_{i,2}} 
 \zeta(\mathfrak u) & = \zeta \Bigl( \sum_{\mathfrak v\in \mathcal V_{i,1}} q(\mathfrak v) - 
 \sum_{\mathfrak u\in \mathcal V_{i,2}} q(\mathfrak u)\Bigr) + O_{\underline y}(f_1^2)\\
 &=  O_{\underline y}(f_1^2).
 \end{align*}
 Thus, $g_2f_1 - g_1 f_2 =O_{\underline y}(f_1^2)$ and the theorem follows.
\end{proof}

\subsection{Proof of Proposition~\ref{claim:final}} 
 Recall that for a partition $\mathcal P$ of $V$ into sets $X_1,\dots, X_k$, we denote by 
 $E(\mathcal P)$ the set of all edges in $G$ with end-points lying  in two different sets among $X_i$s. 
 For a spanning $2$-forest $F$, the partition of $V$ into the vertex sets of the two connected components of  $F$ is as before denoted by
 $\mathcal P(F)$. 

\medskip

Let $\mathcal G$ be a connected component of $\mathfrak G$ 
which does not contain any special vertex. Let $\mathcal V = \mathcal V_1 \sqcup \mathcal V_2$ 
be the vertex set of $\mathcal G$ with $\mathcal V_i \subset \mathfrak V_i$, for $i=1,2$.
We will give a complete description of the structure of $\mathcal G$ using the structure theorem we proved for the exchange 
graph, which in particular allows to prove Proposition~\ref{claim:final}. 
\medskip

Define equivalence relations $\equiv_1, \equiv_2, \equiv_3$ on the set of vertices $V$ of $G$ as follows. For two vertices $u, v \in V$,
\begin{itemize}
 \item  we say $u \equiv_1 v$ if for any $(T_1, T_2, F)\in \mathcal V_2$, 
 both vertices $u$ and $v$ lie in the same connected component of $T_1 \setminus E(\mathcal P(F))$. 
 \end{itemize}
 Similarly, 
 \begin{itemize}
\item we say $u \equiv_2 v$ if for any $(T_1, T_2, F)\in \mathcal V_2$, both vertices $u$ and $v$ lie in the same connected component of $T_2 \setminus E(\mathcal P(F))$. 
\end{itemize}
And finally,
 \begin{itemize}
\item we say $u \equiv_3 v$ if for any $(F_1, F_2, T)\in \mathcal V_1$, both vertices $u$ and $v$ lie in the same connected component of $T \setminus E(\mathcal P(F_1))$. 
\end{itemize}

Note that since $\mathcal G$ does not contain any special vertex, we have $F_1 \sim_v F_2$ for all 
$(F_1, F_2, T)\in \mathcal V_1$. In particular, 
$T \setminus E(\mathcal P(F_1)) = T \setminus E(\mathcal P(F_2))$.

The following statements are analogous to the statements of Lemma~\ref{lem:a1} and Claim~\ref{claim:a2} for the exchange graph. 
\begin{lem}
Let $F$ be a spanning 2-forest in $G$. Let $T$ be a spanning tree of $G$. Suppose two vertices 
$u, v\in V$ are in two different connected components of  $T \setminus E(\mathcal P(F))$. There exists and edge $e\in E(\mathcal P(F)) \,\cap\, E(T_1)$ such that $u$ and $v$ are not connected in $T-e$.
\end{lem}
\begin{proof}
Denote by $S_u$ and $S_v$ the two connected components of $T \setminus E(\mathcal P(F))$ which contain $u$ and $v$, respectively. There is a path joining $S_u$ to $S_v$ in $T$. Since $S_u \neq S_v$, it contains an edge $e \in E(\mathcal P(F))$. For such an edge $e$, $u$ and $v$ are not connected in $T-e$. 
\end{proof}
The previous lemma allows to prove the following claim.
\begin{claim}
 The three equivalence relations $\equiv_1$, $\equiv_2$, $\equiv_3$ are the same. 
\end{claim}
\begin{proof} To prove that $\equiv_1$ and $\equiv_2$ are the same, suppose for the sake of a contradiction that $u \equiv_1 v$ but $u\not\equiv_2 v$ for two vertices $u$ and $v$ in $V$.
 This implies the existence of $(T_1, T_2, F)\in \mathcal V_2$ such that 
\begin{itemize}
\item  the two vertices $u$ and $v$ are both in $X$ with $\mathcal P(F) = \{X,X^c\}$.
\item $u$ and $v$ are in the same connected component of $T_1[X]$, and they are in two different connected components of $T_2[X]$.
\end{itemize} 
Applying the previous lemma, there exists an edge $e\in E(T_2)\cap E(\mathcal P(F))$ such that $u$ and $v$ lie in two different connected components of $F_2 = T_2 -e$. Since 
$(T_1, T_2, F)$ is not special, and $e\in E(\mathcal P(F))$, we have $e\in T_1$. 
In particular, $u,v$ are in the same connected component of $F_1 = T_1 -e$. We have proved that
$\mathcal P(F_1)\neq \mathcal P(F_2)$, i.e., the tuple  $(F_1, F_2, T)$ obtained from $(T_1, T_2, F)$
by pivoting involving $e$ is special. This contradicts the assumption on $\mathcal G$ (that it does not contain special vertices), and proves our claim. 

\medskip
We now prove that $\equiv_1$ and $\equiv_3$ are similar. 
Suppose for the sake of a contradiction that this is not the case.  Let $u,v\in V$ be two vertices with  $u\equiv_3 v$ but $u\not\equiv_1 v$ (the other case $u\not \equiv_3 v$ but $u\equiv_1 v$ has a similar treatment that we omit). This implies the existence of $(T_1, T_2, F) \in \mathcal V_2$ such that $u,v$ belong to two different connected components of 
$T_1 \setminus E(\mathcal P(F))$. Applying the previous lemma, we infer the existence of an edge 
$e\in E(T_1)\,\cap\, E(\mathcal P(F))$ such that $u$ and $v$ are not connected in $T_1-e$. 
Pivoting involving $e$ gives a tuple $(F_1, F_2, T)$ such that $u$ and $v$ lie in two different 
connected components of $F_1$. In particular, it follows that $u \not\equiv_3 v$, which is a contradiction.
This proves the claim.
\end{proof}
We denote by $\equiv$ the equivalence relation on vertices induced by $\equiv_i$. As in Remark~\ref{rem:uf1}, we have the following

\begin{remark}\label{rem:uf}\rm Note that if $u$ and $v$ are two vertices with $u\not \equiv v$, there exists $(F_1, F_2, T) \in \mathcal V_1$ 
such that $u$ and $v$ lie in different connected components of $F_i$. Similarly, there exists 
$(T_1, T_2, F) \in \mathcal V_2$ such that $u,v$ lie in different connected components of $F$. 
\end{remark}

 Denote by $\mathcal P_{\equiv} =\{X_1, \dots, X_n\}$ the partition of $V$ induced by 
 the equivalence classes $X_i$ of $\equiv$. Note that pivoting in $\mathcal G$ 
 only involves edges in $E \setminus E(\mathcal P_\equiv)$, i.e, which are not contained in any 
 $X_1, \dots, X_n$. 
 By connectivity of $\mathcal G$, it follows that for each $i$, there are three trees 
 $\tau_{i,1}, \tau_{i,2}, \tau_{i,3}$ on the vertex set $X_i$ such that for any 
 $(F_1, F_2, T) \in \mathcal V_1$ and any $(T_1, T_2, F) \in \mathcal V_2$, 
 we have 
 $$T_1[X_i] = F_1[X_i] =\tau_{i,1}, \,\, T_2[X_i] = F_2[X_i] =\tau_{i,2},\,\, T[X_i] =F[X_i] =\tau_{i,3}.$$
 In other words, the subtrees $\tau_{i,1}, \tau_{i,2}, \tau_{i,3}$ are the "constant" part of the elements 
 in $\mathcal G$. We now prove
\begin{claim} \label{claim:stab} For any $(T_1, T_2, F)\in \mathcal V_2$, we have 
$$T_1 \setminus \Bigl(\bigcup_{i=1}^n E(\tau_{i,1})\Bigr) = T_2 \setminus \Bigl(\bigcup_{i=1}^n 
E(\tau_{i,2})\Bigr).$$
In other words, the edges of $T_1$ and $T_2$ outside $X_i$'s are the same. 
\end{claim}
\begin{proof}
 Let $e=\{u,v\}$ be an edge of $T_1$ with $u$ and $v$ lying in two different equivalence 
 classes $X_i$ and $X_j$.  By Remark~\ref{rem:uf}, there exists $(T'_1,T'_2, F' ) \in \mathcal V_1$ 
 such that $u$ and $v$  belong to two different  sets of the partition $\mathcal P(F)$. 
 By connectivity of $\mathcal G$, since $e\notin F'$, we have $e\in T_1'$.
 Since $(T'_1,T'_2, F' )$ is non special,  we infer $e \in T_2'$. By connectivity of $\mathcal G$, 
 and the way the edges are defined (which requires pivoting involving the same edge for the two trees 
 in the vertices of $\mathcal V_2$), we must have $e\in T_2$, and the claim follows. 
\end{proof}

Let $\mathfrak v = (T_1, T_2, F) \in \mathcal V_2$. Let 
\begin{align*}E_{1,2}(\mathfrak v) &:= E(T_1) \cap E(\mathcal P_\equiv) = E(T_2) \cap E(\mathcal P_\equiv),\qquad \textrm{and}\\
E_3(\mathfrak v) &:= E(F)\cap E(\mathcal P_\equiv).
\end{align*}
 Obviously, we have 
\begin{align*} 
&E(T_1) = E_{1,2}(\mathfrak v) \sqcup \bigsqcup_{i=1}^n E(\tau_{i,1}) \,\, , \,\, E(T_2) = E_{1,2}(\mathfrak v) \sqcup \bigsqcup_{i=1}^n E(\tau_{i,2})\,\,, \,\, \textrm{and}\\&E(F) = E_3 (\mathfrak v)\cup  \bigcup_{i=1}^n E(\tau_{i,1}).
\end{align*}
Define the multiset 
$$E_{\mathcal G} := E_{1,2}(\mathfrak v) \sqcup E_3(\mathfrak v).$$
 By the definition of the edges in the graph $\mathfrak G$, and connectivity of (the connected component) $\mathcal G$,  
 $E_{\mathcal G}$ is independent of the choice of  $\mathfrak v \in \mathcal V_2$. 
 In addition, if for $\mathfrak u \in \mathcal V_1$, we define $E_{12}(\mathfrak u) = 
 E(F_1) \cap E(\mathcal P_\equiv) $, and $E_{3}(\mathfrak u) = E(T) \cap E(\mathcal P_\equiv)$, 
 we should have  $E_{\mathcal G} = E_{1,2}(\mathfrak u) \sqcup E_3(\mathfrak u)$. 

\medskip

Define an (auxiliary) multigraph $G_0=(V_0, E_0)$ obtained by contracting each equivalence class 
$X_i$ to a vertex $x_i$ and having the multiset of edges $E_0 = E_{\mathcal G}$. 
More precisely, $G_0$ has the vertex set $V_0 = \{x_1, \dots, x_n\}$, and an edge $\{x_i,x_j\}$ 
for any edge $e = \{u,v\}$ in the multiset $E_{\mathcal G}$ which joins a vertex $u\in X_i$ to a vertex 
$v\in X_j$. By an abuse of the notation, we  identify $E_0$ with $E_{\mathcal G}$.

Each $\mathfrak v = (T_1, T_2, F)\in \mathcal V_2$ gives a pair 
$(T_{\mathfrak v}, F_{\mathfrak v})$  that we denote by $\pi(\mathfrak v)$ consisting of a spanning tree 
$T_{\mathfrak v}$ of $G_0$ with edges $E_{1,2}(\mathfrak v)$  and  a spanning 2-forest $F_{\mathfrak v}$ 
of $G_0$ with edge set $E_3(\mathfrak v)$.
 As a multiset, we have $E_0 = E(T_{\mathfrak v}) \sqcup E( F_{\mathfrak v})$.
 Similarly, each $\mathfrak u = (F_1, F_2, T)\in \mathcal V_1$ gives a pair 
 $\pi(\mathfrak u) = (F_{\mathfrak v}, T_{\mathfrak v})$ consisting of a spanning 2-forest 
 $F_{\mathfrak v}$ and a spanning tree $T_{\mathfrak v}$ of $G_0$ with edge sets $E_{1,2}(\mathfrak u)$ 
 and  $E_3(\mathfrak u)$, respectively.

\medskip

We will describe $\mathcal G$ in terms of the multigraph $G_0$.  Let $\mathscr H_0 = 
(\mathscr V_0, \mathscr E_0)$ be the exchange graph associated to the multigraph $G_0$ as in 
Section~\ref{sec:exchange}.
Recall that  the vertex set  $\mathscr V_0$ of $\mathscr H_0$ is the disjoint union of 
two sets $\mathscr V_{0,1}$ and $\mathscr V_{0,2}$, where 
$$\mathscr V_{0,1} := \Bigl\{\, (F, T) \,\big|\,\, F \in \forest(G_0), T\in \tree(G_0), \,\,  E(F) \sqcup E(T) = E_0\,\Bigr\},$$ and 
$$\mathscr V_{0,2} := \Bigl\{\,(T, F) \,\big|\,\, T \in \tree(G_0), F\in \forest(G_0), \,\,  E(F) \sqcup E(T) = E_0\,\Bigr\}.$$
There is an edge in $\mathscr E_0$ connecting $(F, T)\in \mathscr V_{0,1}$ to $(T', F')\in\mathscr V_{0,2}$  if $(T', F')$ \emph{is obtained from $(F, T)$ by pivoting involving an edge} $e\in E_0$,  i.e., if $F  =  T' - e$ and $F' = T -e$.

\medskip

With this notation, we get an application $\pi: \mathcal V \rightarrow \mathscr V_0$. By what we have proved so far, it is clear that $\pi$ is injective. By the definition of edges in $\mathfrak G$ and $\mathscr H_0$, $\pi$ induces a homomorphism of graphs $\pi: \mathcal G \rightarrow \mathscr H_0$. in addition, any pivoting in $\mathscr H_0$ involving an edge $e \in E_0 = E_{\mathcal G}$ 
can be lifted to pivoting involving the same edge $e$ in $\mathcal G$. This proves that $\pi$ induces an isomorphism onto (its image) a connected component of $\mathscr H_0$.

\begin{prop}\label{claim:conn} \label{claim:iso}
The exchange graph $\mathscr H_0$ is connected. As a consequence, the projection map $\pi$ is an isomorphism. 
\end{prop}
\begin{proof} 
By the discussion preceding the proposition,  we only need to show that $\mathscr H_0$ is connected. 
Since the multigraph $G_0$ is disjoint union of a spanning tree and a spanning forest, 
this latter statement  follows from the first part of Theorem~\ref{thm:conn} by observing that 
the only saturated non-empty subset of vertices of $G_0$ are 
singletons. To see this, note that no pivoting involves 
the edge set of a saturated subset of vertices. So for a saturated subset of vertices $S$ of $V_0$, all the components $X_i$ of the partition $\mathcal P_\equiv$ associated to the vertices $x_i\in S$ must lie in the same equivalence class of $\equiv$. However, $X_i \in \mathcal P_\equiv$ are already all the equivalence classes of $\equiv$, so we must have 
$|S|=1$. 
\end{proof}

We can now prove Proposition~\ref{claim:final}.

\begin{proof}[Proof of Proposition~\ref{claim:final}] Let $\mathcal G=(\mathcal V, \mathcal E)$ be a connected component of $\mathfrak G$. Let $G_0$ be the multigraph we associated to $\mathcal G$, and $\pi: \mathcal G \rightarrow \mathscr H_0$ be the isomorphism constructed above. 

\medskip

For $(F, T) \in \mathscr V_{0,1}$, we have $(T, F) \in \mathscr V_{0,2}$, and by definition, we have 
\[q(\pi^{-1}(F, T)) =q (\pi^{-1}(T,  F)).\]
Since $\pi$ is an isomorphism, it follows that 
\begin{align*}
 \sum_{\mathfrak u \in \mathcal V_2} 
q(\mathfrak u)&=\sum_{(T, F) \in \mathscr V_{0,2}}\, q(\pi^{-1}(T, F))\,\,\, = \sum_{(F, T) \in 
 \mathscr V_{0,1}}\, q(\pi^{-1}(F, T))=\sum_{\mathfrak u \in \mathcal V_1} 
q(\mathfrak u),
\end{align*}
and the proposition follows.  
\end{proof}

The proof of Theorem~\ref{claim:aux} is now complete.

\section{Proof of Theorem~\ref{thm:main1}}\label{sec:app}

 In this section we explain how to derive Theorem~\ref{thm:main1} from Theorem~\ref{claim:aux}. The presentation here is heavily based on the results and notations of~\cite{ABBF}, 
 to which we refer for the missing details.

First we recall the set-up. Let $\Delta$ be a small open disc around the origin in $\mathbb C$, and denote by $\Delta^*=\Delta \setminus \{0\}$ the punctured disk.
Let $S=\Delta^{3g-3}$. Let $C_0$ be a stable curve of arithmetic genus $g$, and let $G =(V,E)$ be the dual graph of $C_0$. Denote by $h\leq g$ the genus of $G$, so we have $h = |E|-|V|+1$. The
versal analytic deformation  of $C_0$ over $S$ is denoted by $\pi : \mathcal C \rightarrow S$. The fibers of $\pi$ are smooth outside a normal crossing divisor
$D=\bigcup_{e \in E} D_e \subset S$, which has irreducible components
indexed by the set of  edges of $G$ (which are in bijection with the singular points of $C_0$). Let $U$ be the
complement of the divisor $D$ in $S$, that we identify with $U=(\Delta^*)^{E}\times
\Delta^{3g-3-|E|}$. Let
\begin{equation}\label{eq:uc}
\widetilde{U} := \mathbb H^E \times \Delta^{3g-3-|E|} \longrightarrow U. 
\end{equation}
be the universal cover of $U$.  The projection map $\widetilde U \to U$ is given by $z_e \mapsto \exp(2\pi i
z_e)$ in the first factors corresponding to the edges of $G$, and is the identity on the remaining factors.

\smallskip

Suppose that we have two collections
$$
\sigma_1=\{ \sigma_{l, 1}\}_{l=1, \ldots, n}, \quad \sigma_2=\{\sigma_{l, 2}\}_{\rl=1, \ldots, n}
$$
of sections $\sigma_{\rl, i} \colon S \to \mathcal{C}$ of $\pi$, for $1\leq l\leq n$ and $i=1,2$. By regularity of $\mathcal C$, these sections cannot pass through double points of $C_0$, and for each $l$, 
$\sigma_{l,i}(S)\cap C_0$ lies in a unique irreducible component $X_{v_\rl}$ of $C_0$, which corresponds to a vertex $v_l$ of the dual graph $G$. We assume that the sections 
$\sigma_{l,1}$ and $\sigma_{l,2}$ are distinct on $C_0$, which implies, after shrinking $S$ if necessary, that ${\sigma_1}$ and $\sigma_2$ are disjoint as well.

\smallskip

 Let $\underline \p_1 = \{\p_{l,1}\}_{l=1}^n \in (\R^D)^{n,0}$ and 
 $\underline \p_2 = \{\p_{l,2}\}_{l=1}^{n}\in (\R^D)^{n,0}$ be two collections of external momenta satisfying 
 the conservation law. 
 Using the labelings of sections and the external momenta, we associate each marked point $\sigma_{l,i}$ with $\p_{l,i} \in \R^{D}$, and denote by
 $\underline \p_1^G=(\p_{v, 1}^G)$ and $\p_2^G=(\p_{v, 2}^G)$  the restriction of $\underline \p_1$ and $\underline \p_2$ to the graph $G$: for each vertex 
$v$ of $G$, the vector  $\p_{v,i}^G$ is the sum of all the momenta $\p_{l,i}$ with $v_l = v$.  In this way, at any point $s \in S$, we get two $\R^{D}$-valued degree zero
divisors on the curve $C_s$ that we denote by $\mathfrak A_s $ and $\mathfrak B_{s}$: they are defined by
$$
\mathfrak A_s := \sum_{l=1}^n \p_{l,1}\sigma_{l,1}(s), \qquad \mathfrak B_s := \sum_{\rl=1}^n \p_{l,2}\sigma_{l,2}(s).
$$ 
This gives us the real valued function on $U$ which sends the point $s$ of $U$ to  $\langle \mathfrak{A}_s, \mathfrak{B}_s \rangle$,
where $\langle .\,, . \rangle$ denotes the archimedean height pairing between $\R^{D}$-valued degree zero divisors, see the introduction and~\cite{ABBF} 
for the definition and the extension to $\R^D$-valued divisors defined by means of the given Minkowski bilinear form. 

We are interested in understanding the behaviour of the function $s\mapsto \langle \mathfrak{A}_s, \mathfrak{B}_s \rangle$ close to the origin $0\in S \setminus U$. This can be carried out using the nilpotent orbit theorem in Hodge 
theory, c.f.~\cite{ABBF}. We can reduce to the case where the external momenta are all integers, and in this case, the divisors $\mathfrak A_s$ and $\mathfrak
B_s$ having integer coefficients at any point $s$,  the archimedean height pairing between $\mathfrak{A}_s$ and $\mathfrak{B}_s$ can be described 
in terms of a  biextension mixed Hodge structure, c.f.~\cite{Hain, ABBF}. Denoting by $H_{\mathfrak{B_s}, \mathfrak{A_s}}$ the  biextension mixed Hodge structure associated to the pair 
$\mathfrak{A}_s$ and $\mathfrak{B}_s$, the family $H_{\oB_{s},\oA_{s}}$ fit together into an admissible variation of mixed Hodge structures. 
An explicit description of the period map for the variation of the biextension mixed Hodge structures $H_{\oB_{s},\oA_{s}}$ was obtained in~\cite{ABBF}. We briefly recall this now.
\smallskip

Fix base points $s_0 \in U$ and $\tilde
s_0 \in \widetilde U$ lying above $s_0$, 
and choose a symplectic  basis 
\begin{equation*}
a_1,\dotsc,a_{g},b_1,\dotsc,b_{g} \in H_1(C_{s_0},\Z) =
A_{0}\oplus B_{0}.
\end{equation*}  
Shrinking $S$ if necessary, the inclusion $C_{s_0} \hookrightarrow \mathcal C$ gives a surjective specialization map
\begin{equation*}
\mathrm{sp} \colon H_1(C_{s_0},\Z) \rightarrow H_1(\sC,\Z) \simeq H_1(C_0,\Z).
\end{equation*} 
Denote by $A\subset H_1(C_{s_0},\Z)$ the subspace spanned by the
vanishing cycles $a_e$, one for each $e\in E$. We have the exact sequence
\begin{equation*}
0\to A \to H_1(C_{s_0}, \Z) \xrightarrow{\mathrm{sp}} H_1(C_0, \Z) \to 0,
\end{equation*}
and we define $A'=A + \mathrm{sp}^{-1}(\bigoplus_{v \in V} H_1(X_v, \Z))\subseteq H_1(C_{s_0}, \Z)$. We have
\begin{equation}
 \label{eq:aux}
H_1(C_{s_0}, \Z)/A' \simeq H_1(G, \Z).
\end{equation}
Changing the symplectic basis if necessary, we suppose that the space of vanishing cycles $A$ is generated by
$a_1,\dots, a_h\in A$, and $b_1, \dots, b_h$ generate $H_1(C_{s_0},
\Z)/A' \simeq H_1(G, \Z)$ as in \eqref{eq:aux}.

\smallskip

For $i=1,2$, let $\Sigma _{i,s}=\{\sigma _{1,i}(s),\dots ,\sigma
_{n,i}(s)\}$, and set $\Sigma_{s} =\Sigma _{1, s}\cup \Sigma _{2, s}$ and
$\Sigma _{i}=\bigcup_{s}\Sigma _{i,s}$. By choosing loops that
do not meet the points in $\Sigma _{s_{0}}$, we 
lift the classes $a_{j}$ and $b_{j}$, $j=1,\dots ,g$ to elements of
$H_{1}(C_{s_{0}}\setminus \Sigma _{s_{0}},\mathbb Z)$. By an abuse of the notation, we denote by $a_{j}$ and $b_{j}$ these new classes as well. This  symplectic basis can be spread out to a basis 
\begin{equation*}
a_{1,{\tilde s}}, \dots, a_{g,{\tilde s}}, b_{1,{\tilde s}}, \dots,
b_{{g,\tilde s}} 
\end{equation*}
of $H_{1}(C_{{\tilde s}}\setminus \Sigma _{{\tilde s}},\Z)$, for any $s \in U$ and any $\tilde s\in
\widetilde U$ over $s$. The elements $a_{i,\tilde s}$ only depend on $s$ and not on $\tilde
s$; we will also denote them by $a_{i,s}$, and if there is no risk of confusion, we drop $\tilde s$, and use simply $a_i$ and $b_i$. 

\smallskip

In addition, we have a collection of $1$-forms
$\{\omega_{i}\}_{i=1, \ldots, 
  g}$ on $\pi ^{-1}(U)\subset \mathcal C$ such that the forms $\{\omega_{{i, s}}\coloneqq
\omega_{i}|_{{C_{s}}}\}_{i=1, \ldots, g}$, for each $s\in
U$, are a  basis  of the
holomorphic differentials on $C_s$ and
\begin{equation} \label{eq7}
 \int_{a_{i, s}}
\omega_{j, s}=\delta_{i, j}. 
\end{equation}
The period matrix for the curve $C_s$
is given by
$\bigl(\,\int_{b_{{i, s}}}\omega_{{j,s}}\,\bigr)$. 

\smallskip

Choose now an integer valued 1-chain
$\gamma_{\oB_{s_0}}$ on $C_{s_0} \setminus \Sigma _{1,s_{0}}$
with $\mathfrak B_{s_{0}}$ as boundary. Adding a linear combination of
the $b_{j}$ if necessary, we further assume that
\begin{equation}
  \label{eq5}
  \langle a_{i},\gamma
_{\oB_{s_{0}}}\rangle=0.
\end{equation}
We spread the class 
$$
[\gamma_{\oB_{s_0}}]\in H_{1}(C_{s_0} \setminus
\Sigma _{1,s_{0}},\Sigma _{2,s_{0}},\Z)
$$ 
of $\gamma_{\mathfrak B_{s_0}}$  to classes $\gamma_{\mathfrak B_{\tilde s}}$. 

\smallskip

Similarly, we obtain
a 1-form $\omega_{\mathfrak A}$ on $\pi
^{-1}(U)\setminus \Sigma _{1}$ such that each restriction
$\omega _{\mathfrak A,s}\coloneqq \omega_{\mathfrak A}|_{C_{s}} $ is a holomorphic
form of the third kind with residue $\oA_s$. Adding to $\omega_{\oA}$ a linear combination of the $\omega_{i}$ if needed, we can suppose that $\omega_{\oA}$ is normalized so that
\begin{equation}\label{eq8}
  \int _{a_{i},s}\omega _{\oA,s}=0,\quad i=1,\dots,g.
\end{equation}

Denote by $\mathrm{{Row}}_g(\C)\simeq \mathbb C^g$ and $\mathrm{{Col}}_g(\mathbb C)\simeq \mathbb C^g$ the $g$-dimensional vector space of row and column matrices, and 
let \begin{equation*}
\widetilde{X}:= \mathbb H_g \times \mathrm{Row}_g(\mathbb C) \times \mathrm{Col}_g(\C) \times \C.
\end{equation*}
We have the following description of the period map from~\cite{ABBF}.
\begin{prop}[\cite{ABBF}]The \textit{period map} of the variation of mixed Hodge
  structures $H_{\mathfrak{B}_{s}, \mathfrak{A}_{s}}$ is given by
\begin{align*}
\widetilde\Phi \colon \widetilde U &\longrightarrow \widetilde X \nonumber \\
\tilde{s} &\longmapsto \Bigl(
\,\bigl(\,\int_{b_{{i, \tilde s}}}\omega_{j,s}\,\bigr)_{i,j}\,,
\, \bigl(\,\int_{\gamma_{\oB, \tilde s}}\omega_{j,s}\,\bigr)_j\,,
\,\bigl(\,\int_{b_{i,\tilde s}}\omega_{\mathfrak A,s}\,\bigr)_i\,,
\,\int_{\gamma_{\mathfrak B, \tilde s}}\omega_{\oA,s}\,\Bigr).
\end{align*}
\end{prop}

\medskip

We now explain the action of the logarithm of monodromy map $N_e$, for $e\in E$, c.f.~\cite{ABBF}.

As before, each
vanishing cycle $a_{e}\in H_{1}(C_{s_{0}},\Z)$  for $e\in E$ can be lifted
in a canonical way to a cycle $a_e$ in $H_{1}(C_{s_{0}}\setminus \Sigma
_{s_{0}},\Z)$. 

In this homology 
group, we can write 
\begin{equation} \label{eq:2}
  a_{e}=\sum_{i}c_{{e,i}} a_{i}+\sum _{l} d_{{e,l,1}} \gamma _{{l,1}} +
  \sum_{l} d_{{e,l,2}} \gamma _{l,2},
\end{equation}
with $\gamma _{l,i}$ denoting a small enough negatively oriented loop around
the point $\sigma _{l,i}({s_{0}})$. Note that the coefficients  $c_{e,i}$ are zero for $i>h$ (by the choice of the symplectic  basis $\{a_{i},b_{i}\}$). 

\smallskip

By Picard-Lefschetz formula, we deduce from~\eqref{eq5} and~\eqref{eq:2} that 
\begin{align}
  N_{e}(b_{i})
  &=-\langle {b_{i},a_{e}}\rangle a_{e} = c_{e,i}a_{e},\label{equation10}\\
  N_{e}(\gamma _{\oB_{s_{0}}})
  &=-\langle \gamma_{\oB_{s_{0}}},a_{e}\rangle a_{e} =
    -a_{e}\sum_{l}\ps_{l,2}d_{e,l,2}.\label{equation11}
\end{align}
Using 
\eqref{eq:2}, \eqref{eq8} and \eqref{eq7}, we can compute the integral of the forms $\omega _{j}$ and $\omega _{\mathfrak A}$ 
with respect to the vanishing cycles, giving
\begin{equation}\label{equation9}
  \int_{a_{e}}\omega _{j}=c_{e,j},\quad \int_{a_{e}}\omega
  _{\oA_{s_{0}}}=\sum_{l} \ps_{l,1}d_{e,l,1}.
\end{equation}

From \eqref{equation10}, \eqref{equation11} and \eqref{equation9}, we get 
\begin{align*}
  N_e(\int_{b_i}\omega_{j,s_{0}})
  &= -\langle b_i,a_e\rangle\int_{a_e}\omega_{j,s} =
    c_{e,i}c_{e,j}\,;\\
    N_e(\int_{b_i}\omega _{\mathfrak A_{s_{0}}})
  &=-\langle b_i,a_e\rangle\int_{a_e}\omega_{\mathfrak A_{s_{0}}}=
    c_{e,i}\sum_{l}\ps_{l,1}d_{e,l,1}\,;\\ 
  N_e(\int_{\gamma_{\mathfrak B_{s_{0}}}}\omega_{j,s_{0}})  
  &= -\langle \gamma_{\mathfrak B_{s_{0}}},a_e\rangle\int_{a_e}\omega_{j,s}=
    -c_{e,j}\sum_{l}\ps_{l,2}d_{e,l,2}\,;\\ 
  N_e(\int_{\gamma_{\mathfrak B_{s_{0}}}}\omega _{\mathfrak A_{s_{0}}})
  &= -\langle \gamma_{\mathfrak B_{s_{0}}},a_e\rangle
    \int_{a_e}\omega_{\mathfrak A_{s_{0}}}=
    -\Big(\sum_{l}\ps_{l,1}d_{e,l,1}\Big)
    \Big(\sum_{k}\ps_{k,2}d_{e,k,2}\Big).
  \end{align*}
  
For each $e\in E$, the logarithm of the monodromy $N_e$ is given by 
\begin{equation*}
N_{e}=
  \begin{pmatrix} 0 & 0 & \underline{\ps}_{2}\widetilde{W}_{e} &
    \underline{\ps}_{2}\Gamma _{{e}} {^{t}\underline{\ps}_{1}}  \\0 &
    0 & \widetilde{M}_{e} & \widetilde {Z}_{e} {^{t}\underline{\ps}_{1}}\\
    0 & 0 & 0 & 0 \\ 0 & 0 & 0 & 0\end{pmatrix}, 
\end{equation*}
where the matrices $\widetilde M_{e}$, $\widetilde W_{e}$,
$\widetilde Z_{e}$, and
$\Gamma _{e}$ are given by
\begin{equation*}
 (\widetilde M_{e})_{i,j}=c_{e,i}c_{e,j},\qquad
  (\widetilde W_{e})_{l,j}=-c_{e,j}d_{e,l,2},\qquad (\widetilde
  Z_{e})_{i,l}=c_{e,i}d_{e,l,1},\qquad  
  (\Gamma _{k,l})=-d_{e,k,2}d_{e,l,1}.
\end{equation*}

\smallskip

One verifies that the matrix $\widetilde M_{e}$ is the 
$h\times h$ matrix $M_{e}$ filled with zeros to a $g \times g$
matrix, where $M_e$ is the matrix of the symmetric bilinear form $\langle . \rangle_e$ in the basis $b_1, \dots, b_h$ of $H_1(G, \mathbb Z)$.   Similarly, one sees that the matrix $\widetilde W_{e}$ (resp. $\widetilde{Z}_{e}$) is obtained from a matrix $W_{e}$ (resp. $Z_e$) that has only $h$ columns (resp. rows)  by extension with zeros.
 The entries of these matrices are given as follows. The choice of the path $\gamma _{\mathfrak B}$ provides a preimage $\omega
_{2}$ for the vector $\mathbf p_{2}^{G}$ in $\Z^{E}$, obtained by counting the number of
times with sign that $\gamma _{\mathfrak B}$ crosses the vanishing cycle
$a_{e}$. Similarly, $\omega _{\mathfrak A}$ gives a preimage
$\omega _{1}$ for $\mathbf p_{1}^{G}$ in $\mathbb C^{E}$ whose $e$-th component, for $e\in E(G)$, is given
by  $\int_{a_{e}}\omega _{\mathfrak A}.$

With these preliminaries, we can now state the expression of the height pairing in terms of the period map. Let us separate the variables which correspond to the edges of the graph $G$ as $
 s_E$.  Any point $s$ of $U$ then can be written as $s =
 s_E \times s_{E^c}$, where $s_{E^c}$ denotes all the other $3g-3-|E|$ coordinates. Denoting the coordinates in the
 universal cover $\widetilde U$ by $z_{e}$, the projection $\widetilde U\to U$ is
 given in these coordinates by
 \begin{equation*}
   s_{e}=
   \begin{cases}
     \exp(2\pi i z_{e}),&\text{ for }e\in E,\\
     z_{e},&\text{ for }e\not \in E.
   \end{cases}
 \end{equation*}
 
 The following
expression for the height pairing function is obtained in~\cite{ABBF}.
\smallskip

\begin{prop}[\cite{ABBF}]\label{prop:2} There exists $h_0>0$ and a holomorphic map $\Psi_0: U \rightarrow \widetilde X$,
\begin{equation*}
  \Psi _0(s)=(\Omega _{0}(s),
  W_{0}(s),Z_{0}(s),\rho _{0}(s)),
\end{equation*}
such that introducing
\begin{equation*}
  y_{e}=\Im(z_{e})=\frac{-1}{2\pi }\log|s_{e}|,
\end{equation*}
the height pairing is given by
  \begin{multline}\label{eq:prop57}
    \langle \oA_{s},\oB_{s}\rangle=     -2\pi \Im(\rho _{0})-\sum_{e\in E}2\pi    y'_{e}\underline{\ps}_{2}\Gamma_{e} \,{^{t}\underline{\ps}_{1}}+\\
    2\pi \Big(\Im(W_{0})+\sum_{e\in
      E}y'_{e}\underline{\ps}_{2}\widetilde W_{e}\Big)\cdot 
    \Big(\Im(\Omega _{0})+\sum_{e\in E}y'_{e}\widetilde M_{e}\Big)^{-1}\\
   \cdot \Big(\Im(Z_{0})+\sum_{e\in
      E}y'_{e}\widetilde Z_{e}\,{^{t}\underline{\ps}_{1}}\Big), 
  \end{multline}
  where $y_{e}'=y_{e}-h_{0}$.
\end{prop}

We have

\begin{thm}\label{thm:4} There exists a bounded function $h\colon U\to \R$ such that after shrinking the radius of $\Delta $ if necessary, we can write the height pairing as 
  \begin{multline}
    \label{eqfinal}\langle \mathfrak A_{s},\mathfrak B_{s}\rangle=
    -\sum_{e\in E}2\pi y_{e}\underline{\ps}_{2}\Gamma
    _{e} \,{^{t}\underline{\mathbf p}_{1}}
    +2\pi \Big(\sum_{e\in E}y_{e}\underline{\mathbf p}_{2}W_{e}\Big)
    \Big(\sum_{e\in E}y_{e} M_{e}\Big)^{-1}
    \Big(\sum_{e\in
      E}y_{e}Z_{e}\,{^{t}\underline{\mathbf p}_{1}}\Big)+h(s).
  \end{multline}
\end{thm}

 This theorem was proved in~\cite{ABBF} using normlike functions in the terminology of \cite[Section
 3.1]{BHJ}. We now give a proof based on Theorem~\ref{claim:aux}.

\medskip

  Since $\rho _{0}$ is a holomorphic function on $S=\Delta ^{3g-3}$, after
  shrinking the radius of $\Delta $ if necessary, we can assume that
  $\Im(\rho _{0})$ is bounded. In addition, since $h_0$ is constant, the difference between 
  $y'_{e}\underline{\ps}_{2}\Gamma_{e} \,{^{t}\underline{\ps}_{1}}$ and $y_{e}\underline{\ps}_{2}\Gamma_{e} \,{^{t}\underline{\ps}_{1}}$ is constant for each $e$. So we only need to prove that the third
  term in the right hand side of equation \eqref{eq:prop57} is, up to a bounded function, equal to the second term in the right hand side of \eqref{eqfinal}. 

\smallskip

We treat first the case $g=h$  and explain later how to reduce to this case. 
 
 \medskip

First, we can reduce to the case where $\p_i$ are real valued, c.f.~\cite{ABBF}.  Using the bilinearity of the right hand side term in~\eqref{eqfinal}, we can reduce to the case $\p_1 =\p_2$.

 Let $H = H_1(G, \R)$, $\omega \in \R^E$ given by $\p$, and $H_\omega \supset H$ as in Section~\ref{sec:mainthmintro}.  Let $\alpha = \sum_e y_e \langle\,,\rangle_e$ the bilinear form on $\R^E$. 

For a matrix of the form, 
\begin{equation*}
T= \begin{pmatrix} M & W\\
\prescript{\mathrm t}{}{W} &  S
\end{pmatrix}
\end{equation*}
where $M$ is an invertible $h\times h$ matrix, $W$ is a (column) vector of dimension $h$, and $S$ is a scalar, we have the formula 
$(\det M)M^{-1} = \text{adj}(M)$, where $\text{adj}(M)$ is the matrix of minors. This gives 
\begin{equation*}
\frac{\det T}{\det M}= -\prescript{\mathrm t}{}{W}M^{-1}W+S.
\end{equation*}

\smallskip

Using these observations, the expression on the right hand side of~\eqref{eqfinal} is the ratio $2\pi\det(\alpha|_{H_\omega})/{\det(\alpha|_{H})}$, for the basis of $H$ (resp.) given by $B=\{b_1, \dots, b_h\}$ (resp. $B_\omega=\{b_1, \dots, b_h, \omega\}$). Similarly,
the expression      on the right hand side of Proposition~\ref{prop:2} at any point $s$ of $U$ is the ratio ${2\pi\det(\alpha|_{H_\omega} + \beta(s)|_{H_\omega})}/{\det(\alpha|_{H} + \beta(s)|_{H})}$ for a bilinear form $\beta(s)$ 
     on $H_\omega$ (given by $W_0$, $Z_0$, $\Omega_0$, $h_0$, and $\Gamma_e, W_e, Z_e, M_e$), calculated using the basis $B$ and $B_\omega$ of $H$ and $H_\omega$). 
 
 By boundedness of $W_0, Z_0, \Omega_0,$ and $h_0$, $\beta(s)$ lies in a compact subset of the space of bilinear forms on $H_\omega$. 
 Fixing a complmenet $H'$ to $H_\omega$, i.e., $H_\omega+H' =\R^m$, and extending $\beta(s)$ trivially (by zero) to $\R^m$, we can assume that $\beta(s)$ is the restriction to $H_\omega$ of a bilinear form  
 $\widetilde \beta(s)$ on $\R^m$, and that $\widetilde \beta(s)$ lie in a compact subset of the space of bilinear forms on $\R^m$ for $s\in U$.
 
 \medskip
 
Let $M$ (resp. $N$) be the $h\times m$ (resp. $(h+1)\times m$) matrix of the coefficients of the basis $B$ (resp. $B_\omega$) 
 in the standard basis of $\R^m$.  Let $Y = \mathrm{diag}(y_1, \dots, y_m)$ be 
 the diagonal $m\times m$ matrix of $\alpha$ in the standard basis of $\R^m$. 

\medskip

Let $A: U \rightarrow \mathrm{Mat}_{m\times m}(\R)$ be the matrix-valued map taking at $s\in U$ the value $A(s)$ the matrix of the bilinear form $\widetilde \beta(s)$ 
 in the standard basis of $\R^m$. 

\medskip

 Theorem~\ref{thm:4} in the case $g=h$ now follows from Theorem~\ref{thm:main1}, which is  the statement  that the difference 
 ${\det(N (Y + A)N^\tau)}/{\det (M(Y+A)M^\tau)} - {\det(N Y N^\tau)}/{\det(MYM^\tau)}$ is  $O_{\underline y} (1)$.

\vspace{.7cm}
We now explain how to reduce the general case $g>h$ to the case treated above.

  \medskip
  
  Let $\mathscr W := \Im(W_{0})-\sum_{e\in
      E}y_{e}\underline{\ps}_{2}\widetilde W_{e}$, and write $\mathscr W = (\mathscr W_1, \mathscr W_2)$ with $\mathscr W_1$ the vector of the $h$ first  coordinates.
       Similarly, write 
    $\mathscr Z= \Im(Z_{0})+\sum_{e\in
      E}y_{e}\widetilde Z_{e}\,{^{t}\underline{\ps}_{1}}$, and write $\mathscr Z = {^{t}(\mathscr Z_1, \mathscr Z_2)}$ with $\mathscr Z_1$ the first vector of the $h$ first coordinates. 
    
    Let $\mathscr M = \Im(\Omega_0)+\sum_{e\in E}y_{e}\widetilde M_{e}$,  and write 
    \begin{equation}\label{eq:A}
\mathscr M=
  \begin{pmatrix} \mathscr M_{11} & \mathscr M_{12}\\
   \mathscr M_{21}& \mathscr M_{22}. 
   \end{pmatrix}
\end{equation}
It will be enough to prove 
\begin{prop}\label{claim:1}
We have
$$ \mathscr W \,\mathscr M^{-1} \mathscr Z - \Big(\sum_{e\in E}y_{e}\underline{\ps}_{2}W_{e}\Big)
    \Big(\sum_{e\in E}y_{e} M_{e}\Big)^{-1}
    \Big(\sum_{e\in
      E}y_{e}Z_{e}\,{^{t}\underline{\ps}_{1}}\Big) = O_{\underline y}(1).$$ 
\end{prop}
\begin{proof}
  Let $\mathscr N =\mathscr M^{-1}$, and write
  \begin{equation}\label{eq:A}
\mathscr N=
  \begin{pmatrix} \mathscr N_{11} & \mathscr N_{12}\\
   \mathscr N_{21}& \mathscr N_{22}
   \end{pmatrix}
\end{equation}
 with $\mathscr N_{11}$ and $\mathscr N_{22}$ square matrices of size $h\times h$ and $(g-h)\times(g-h)$, respectively. 
  Writing 
  \[\mathscr W \mathscr N \mathscr Z = \mathscr W_1\,\mathscr N_{11} \mathscr Z_1 +\,\mathscr W_1 \,\mathscr N_{12} \mathscr Z_2\,+ \mathscr W_2 \,\mathscr N_{21} \mathscr Z_1\,+ 
  \mathscr W_2 \,\mathscr N_{22} \mathscr Z_2,\]
  in order to prove Claim~\ref{claim:1}, we prove 
$$ \mathscr W_1 \,\mathscr N_{11} \mathscr Z_1 - \Big(\sum_{e\in E}y_{e}\underline{\ps}_{2}W_{e}\Big)
    \Big(\sum_{e\in E}y_{e} M_{e}\Big)^{-1}
    \Big(\sum_{e\in
      E}y_{e}Z_{e}\,{^{t}\underline{\ps}_{1}}\Big) = O_{\underline y}(1),$$
      and  
      $$\mathscr W_1 \,\mathscr N_{12} \mathscr Z_2 = O_{\underline y}(1), \quad \mathscr W_2 \,\mathscr N_{21} \mathscr Z_1 =O_{\underline y}(1), \quad  \mathscr W_2 \,\mathscr N_{22} \mathscr Z_2=O_{\underline y}(1).$$
 For $y_1, \dots, y_m$ large enough, since $\mathscr M_{12}, \mathscr M_{21}, 
\mathscr M_{22}$ are bounded, we have the following expressions:
\[\mathscr N_{11} =  (\mathscr M_{11} - \mathscr M_{12} \mathscr M_{22}^{-1} \mathscr M_{21})^{-1} \qquad ,\qquad  \mathscr N_{22} =  (\mathscr M_{22} - \mathscr M_{21} \mathscr M_{11}^{-1} \mathscr M_{12})^{-1}\]
\[\mathscr N_{12} = - \mathscr M_{11}^{-1}\mathscr M_{12} (\mathscr M_{22} -\mathscr M_{21}\mathscr M_{11}^{-1}\mathscr M_{12})^{-1}, \qquad \textrm{and}\]
\[\mathscr N_{21} = - \mathscr M_{22}^{-1}\mathscr M_{21} (\mathscr M_{11} -\mathscr M_{12}\mathscr M_{22}^{-1}\mathscr M_{21})^{-1}.\]

 \medskip
 Note that $\mathscr M_{22}(s) = \Omega_0^{22}(s)$ for $s\in U$, and by our assumption on $U$, the matrices $\mathscr M_{22}^{-1}(s)$ lies in a compact set for $s\in U$.
 Thus, $\mathscr N_{11}  = \mathscr A(s)+ \sum_{e} y_eM_e$ for an $h\times h$ matrix-valued map $\mathscr A$ on $U$ taking values in a compact set  provided that $y_1, \dots, y_m$ are large. 
 It follows from the result in the case $g=h$ that 
 $$ \mathscr W_1 \,\mathscr N_{11} \mathscr Z_1 - \Big(\sum_{e\in E}y_{e}\underline{\ps}_{2}W_{e}\Big)
    \Big(\sum_{e\in E}y_{e} M_{e}\Big)^{-1}
    \Big(\sum_{e\in
      E}y_{e}Z_{e}\,{^{t}\underline{\ps}_{1}}\Big) = O_{\underline y}(1).$$

      The boundedness of the other three quantities can be proved similarly. For example, to treat the term $\mathscr W_1 \,\mathscr N_{12} \mathscr Z_2$, we observe first that 
      $\mathscr C = \mathscr M_{12} (\mathscr M_{22} -\mathscr M_{21}\mathscr M_{11}^{-1}\mathscr M_{12})^{-1}$ lies in a bounded compact set provided that $y_1, \dots, y_m$ are large enough. 
    We have 
    \[\mathscr W_1 \,\mathscr N_{12} \mathscr Z_2 = - \mathscr W_1 \mathscr M_{11}^{-1} \mathscr C = - \mathscr W_1 \mathscr M_{11}^{-1} \Bigl(\mathscr C_1 -\mathscr C_2\Bigr),\]
    with $\mathscr C_2 = \sum_{e\in
      E}y_{e}Z_{e}\,{^{t}\underline{\ps}_{1}}$ and $\mathscr C_1 =  \mathscr C + \mathscr C_2$.  
      
      Applying the result in the case $g=h$, we have for both the quantities for $k=1,2$ 
    $$ \mathscr W_1 \,\mathscr N_{11} \mathscr C_k - \Big(\sum_{e\in E}y_{e}\underline{\ps}_{2}W_{e}\Big)
    \Big(\sum_{e\in E}y_{e} M_{e}\Big)^{-1}
    \Big(\sum_{e\in
      E}y_{e}Z_{e}\,{^{t}\underline{\ps}_{1}}\Big) = O_{\underline y}(1).$$
    Taking now their difference shows what we wanted to prove. 
\end{proof}
 This finishes the proof of Theorem~\ref{thm:4}.     To prove Theorem~\ref{thm:main1}, we remark that by~\cite{ABBF}, the expression on the right hand side of Theorem~\ref{thm:4} is precisely the right hand side term in Theorem~\ref{thm:main1}.


\begin{thebibliography}{99}

\bibitem{ABBF} O.~Amini, S.~Bloch, J.~I.~Burgos Gil, J.~Fresan, Feynman amplitudes and limits of heights, 
\textit{Izvestiya: Mathematics 80 (2016), no. 5, special issue in honor of J-P. Serre, to appear.} 


\bibitem{BHJ} J.~I.~Burgos Gil, R.~de~Jong, D.~Holmes, Singularity of
  the biextension metric for families of abelian varieties, preprint
  arxiv:1604.00686.
\bibitem{Chaiken} S. Chaiken, A combinatorial proof of the all minors matrix tree theorem, \textit{SIAM J. Algebraic Discrete Methods} \textbf{3} (1982), no. 3, 319--329.

\bibitem{ELOP}
R.~J.~Eden, P.~V.~Landshoff, D.~I.~Olive, and J.~C.~Polkinghorne, \textit{The Analytic $S$-Matrix}, Cambridge University Press, London-New York-Ibadan, 1966.


\bibitem{Hain} R.~Hain, Biextensions and heights associated to curves of odd genus, Duke Math. J. \textbf{61} (1990), no. 3, 859--898.

\bibitem{IZ} 
C.~Itzykson and J.-B.~Zuber, \textit{Quantum Field Theory}, International Series in Pure and Applied Physics, McGraw-Hill International, New York, 1980. 

\bibitem{Kir} G. Kirchhoff, \"Uber die Aufl\"{o}sung der Gleichungen, auf welche man bei der Untersuchung der linearen Vertheilung galvanischer Str\"ome gef\"uhrt wird, \textit{Annalen der Physik} \textbf{148} (1847), no. 12, 497--508.


\end{thebibliography}
\end{document}